\documentclass[12pt]{article}
\usepackage{amssymb}
\usepackage{graphicx}
\usepackage[all]{xy}
\usepackage{color}
\usepackage{latexsym}
\usepackage{amssymb,amsmath,amstext}
\usepackage{epsfig}
\usepackage{graphics}
\usepackage{subfigure}
\usepackage{euscript}
\usepackage{ifthen}
\usepackage{wrapfig}
\usepackage{amsthm}
\usepackage{multicol}
\usepackage{paralist}
\usepackage{verbatim}
\usepackage[colorlinks=true, urlcolor=blue, linkcolor=blue, citecolor=blue]{hyperref}
\usepackage[margin=1in]{geometry}

\numberwithin{equation}{section}
\theoremstyle{plain}%
\newtheorem{theorem}{Theorem}
\numberwithin{theorem}{section}
\newtheorem{proposition}[theorem]{Proposition}
\newtheorem{example}[theorem]{Example}
\newtheorem{lemma}[theorem]{Lemma}
\newtheorem{corollary}[theorem]{Corollary}
\newtheorem{definition}[theorem]{Definition}
\newtheorem{remark}[theorem]{Remark}
\newtheorem{conjecture}[theorem]{Conjecture}

\newcommand{\C}{\mathbb{C}}

\newcommand{\PP}{\mathbb{P}}

\newcommand{\R}{\mathbb{R}}

\newcommand{\sH}{\mathcal{H}}

\newcommand{\sL}{\mathcal{L}}

\allowdisplaybreaks

\date{}
\begin{document}

\title{\bf Likelihood Geometry}

\author{June Huh and Bernd Sturmfels}

\maketitle

 \begin{abstract}
 \noindent
We study the critical points of monomial functions over an algebraic subset of the 
probability simplex. The number of critical points on the Zariski closure
is a topological invariant of that embedded projective variety, known as its 
maximum likelihood degree. We present an introduction to this theory and
its statistical motivations.  Many favorite objects from
combinatorial algebraic geometry are featured: toric varieties,
$A$-discriminants, hyperplane arrangements,
Grassmannians, and determinantal varieties.
Several new results are included, especially on the likelihood correspondence and its bidegree.
This article represents the lectures given by the second author
at the CIME-CIRM course on Combinatorial Algebraic Geometry
at Levico Terme in June~2013.
         \end{abstract}

 \section*{Introduction}

Maximum likelihood estimation (MLE) is a fundamental computational problem
in statistics, and it has recently been studied with some success from the perspective
of algebraic geometry. In these notes we give an introduction to the geometry
behind MLE for algebraic statistical models for discrete data. As is customary
in algebraic statistics \cite{LiAS}, we shall identify such models with certain algebraic
subvarieties of high-dimensional complex projective spaces.

The article is organized into four sections.
The first three sections correspond to the three lectures given at Levico Terme.
The last section will contain proofs of new results.

In Section 1, we start out with plane curves,
and we explain how to identify the 
relevant punctured Riemann surfaces.
We next present the definitions 
and basic results for likelihood geometry in $\mathbb{P}^n$.
Theorems \ref{thm:finite-to-one}
and \ref{thm:eulerchar} are
concerned with the likelihood correspondence, the sheaf of
differential $1$-forms with logarithmic poles, and the topological Euler characteristic.
The  ML degree of generic complete intersections is given in
Theorem \ref{thm:genericMLdegree}.
Theorem \ref{PositiveData} shows that the likelihood
fibration behaves well over strictly positive data.
Examples of Grassmannians and Segre varieties 
are discussed in detail.
Our treatment of linear spaces in Theorem \ref{LinearCase}
will appeal to readers
interested in matroids and hyperplane arrangements.

Section 2 begins leisurely, with the question
{\em Does watching soccer on TV cause hair loss?} \cite{MSS}.
This leads us to conditional independence and 
low rank matrices. We study  likelihood geometry
of determinantal varieties, culminating in the
duality theorem of Draisma and Rodriguez \cite{DR}.
The ML degrees in Theorems \ref{thm:MLvalues}
and \ref{thm:MLSymmValues} were computed
using the software {\tt Bertini} \cite{Bertini}, underscoring the
benefits of using numerical algebraic geometry for~MLE.
After a discussion of mixture models, highlighting
the distinction between rank and nonnegative rank, we end Section 2
with a review of recent results in \cite{ARSZ} on tensors of nonnegative rank~$2$.

Section 3 starts out with toric models \cite[\S 1.22]{PS}
 and geometric programming \cite[\S 4.5]{BoydVan}.
Theorem \ref{thm:toricMLE} identifies
the ML degree of a toric variety with the Euler characteristic of 
the complement of a  hypersurface in a torus.
Theorem \ref{thm:genericmap} 
furnishes the ML degree of a variety
parametrized by generic polynomials.
Theorem \ref{thm:huh2} characterizes
varieties of ML degree $1$ and it
reveals a beautiful connection to the
$A$-discriminant of \cite{GKZ}.
We introduce the ML bidegree
and the sectional ML degree of
an arbitrary projective variety in $\PP^n$,
and we explain how these two are related.
Section 3 ends with a study of the
operations of  intersection, projection, and restriction
in likelihood geometry. This concerns the algebro-geometric meaning of
the distinction between sampling zeros 
and structural zeros in statistical modeling.

In Section 4 we offer precise definitions and technical explanations
of more advanced concepts from algebraic geometry, including
logarithmic differential forms,  Chern-Schwartz-MacPherson 
classes, and sch\"on very affine varieties. This enables us to present
complete proofs of various results, both old and new,
that are stated in the earlier sections.

\smallskip

We close the introduction with a disclaimer
regarding our overly ambitious title. There are many important topics
in the statistical study of likelihood inference
that should belong to ``Likelihood Geometry''  but are not covered in this article.
Such topics include Watanabe's theory of  singular Bayesian integrals \cite{Wat}, 
differential geometry of likelihood in information geometry \cite{AN},
and real algebraic  geometry of  Gaussian models \cite{Uhl}.
We regret not being able to talk about these topics and many others.
Our presentation here is restricted to  the setting of \cite[\S 2.2]{LiAS}, namely
statistical models for discrete data viewed as projective varieties in $\PP^n$.

\section{First Lecture}

Let us begin our discussion with likelihood on algebraic curves
 in the complex projective plane $\PP^2$.
We fix a system of homogeneous coordinates $p_0,p_1,p_2$  on $\PP^2$.
The set of real points in $\PP^2$ with
${\rm sign}(p_0) = {\rm sign}(p_1) = {\rm sign}(p_2)$ is
identified with the open triangle
$$ \Delta_2 \,\, = \,\, \bigl\{ \,(p_0,p_1,p_2) \in \R^3 \,:\,
p_0,p_1,p_2 > 0 \,\,\,\hbox{and} \,\, \, p_0 +  p_1 + p_2 = 1 \,\bigr\}.$$
Given three positive integers $u_0,u_1,u_2$, the corresponding likelihood function is
$$ \ell_{u_0,u_1,u_2}(p_0,p_1,p_2) \,\, = \,\, \frac{p_0^{u_0} p_1^{u_1} p_2^{u_2}}{(p_0+p_1+p_2)^{u_0+u_1+u_2}}. $$
This defines a rational function on $\PP^2$, and it restricts to a regular function
on $\PP^2 \backslash \mathcal{H}$, where
$$ \mathcal{H} \,\,=\, \bigl\{\, (p_0:p_1:p_2) \in \PP^2 \,: \, p_0p_1p_2(p_0+p_1+p_2) = 0 \,\bigr\} $$
is our arrangement of four distinguished lines.
The likelihood function $\ell_{u_0,u_1,u_2}$ is positive on the triangle $\Delta_2$,
it is zero on the boundary of $\Delta_2$, 
and it attains its maximum at the point 
\begin{equation}
\label{eq:mlexyz} (\hat p_0, \hat p_1, \hat p_2)  \,\, = \,\, \frac{1}{u_0+u_1+u_2}(u_0,u_1,u_2). 
\end{equation}
The corresponding point $(\hat p_0:\hat p_1:\hat p_2)$ is the only 
critical point of the function $\ell_{u_0,u_1,u_2}$ on the 
four-dimensional real manifold $\PP^2 \backslash \mathcal{H}$.
To see this, we consider the {\em logarithmic derivative}
$$ {\rm dlog}(\ell_{u_0,u_1,u_2}) \,\, = \,\,
\biggl(
 \frac{u_0}{p_0} - \frac{u_0+u_1+u_2}{p_0+p_1+p_2}\,,\,\, 
\frac{u_1}{p_1} - \frac{u_0+u_1+u_2}{p_0+p_1+p_2}\,,\,\,
\frac{u_2}{p_2} - \frac{u_0+u_1+u_2}{p_0+p_1+p_2}
\biggr).
$$
We note that this equals $(0,0,0)$
if and only if $(p_0:p_1:p_2)$ is the point $(\hat p_0:\hat p_1:\hat p_2)$  in (\ref{eq:mlexyz}).

\smallskip

Let $X$ be a smooth curve in $\PP^2$ defined by a homogeneous polynomial $f(p_0,p_1,p_2)$.
This curve plays the role of a statistical model,
and our task is to maximize the likelihood function
$\ell_{u_0,u_1,u_2}$ over its set  $X \cap \Delta_2$
of positive real points. To compute that maximum
algebraically, we examine the set of all
critical points of $\ell_{u_0,u_1,u_2}$ on the
complex curve $X \backslash \mathcal{H}$.
That set of critical points is the {\em likelihood locus}. Using Lagrange Multipliers
from Calculus, we see that it  consists of
all points of $X \backslash \mathcal{H}$ such that
${\rm dlog}(\ell_{u_0,u_1,u_2})$ lies in the plane
spanned by ${\rm d}f$ and $(1,1,1)$ in $\C^3$.
Thus, our task is to study the solutions in $\PP^2\backslash \mathcal{H}$
of the equations
\begin{equation}
\label{eq:planefJ}
f(p_0,p_1,p_2) \,=\,0 \qquad \hbox{and} \qquad
{\rm det}
\begin{pmatrix}
  1 & 1 & 1  \smallskip \\
  \frac{u_0}{p_0} & \frac{u_1}{p_1} & \frac{u_2}{p_2}  \medskip \\
   \frac{\partial f}{\partial p_0} &  \frac{\partial f}{\partial p_1} &  \frac{\partial f}{\partial p_2}
   \end{pmatrix}\,=\,\, 0.
 \end{equation}
      Suppose that $X$ has degree $d$. Then,  after clearing denominators, 
      the second equation has degree $d+1$.  
      By B\'ezout's Theorem, we expect the
      likelihood locus to consist of $d(d+1)$ points in $\PP^2 \backslash \mathcal{H}$.
      This is indeed what happens when $f$ is a generic polynomial of degree~$d$.

We define the {\em maximum likelihood degree} (or {\em ML degree}) of our curve $X$
to be the cardinality of the likelihood locus for generic choices of $u_0,u_1,u_2$.
Thus a general plane curve of degree $d$ has ML degree $d(d+1)$.
However, for special curves, the ML degree can be smaller.

\begin{theorem} \label{thm:RiemannSurface}
Let $X $ be a smooth curve
of degree $d$ in $\PP^2$, and 
$a= \#(X \cap \mathcal{H})$ the
 number of its points on the
distinguished arrangement.
Then the ML degree of $X$ equals
$d^2-3d+a$.
\end{theorem}

This  is a very special case of
Theorem  \ref{thm:eulerchar} which identifies the ML degree with the signed
Euler characteristic of $X \backslash \mathcal{H}$.
For a general curve of degree $d$ in $\PP^2$, we have
$ a = 4 d$, and so $d^2 - 3d + a = d(d+1)$
as predicted. However, the number $a$ of points in $X \cap \mathcal{H}$ can drop:

\begin{example} 
\label{ex:linesinplane}
\rm Consider the case $d=1$ of lines.
A generic line has ML degree $2$.
The line $X = V( p_0 + c p_1 )$ has ML degree $1$
provided $c \not\in \{0,1\}$.
The special line $X = V(p_0 + p_1 )$ has ML degree $0$:
 (\ref{eq:planefJ}) has no solutions
on $X \backslash \mathcal{H}$ unless $u_0+u_1 = 0$.
In the three cases, $X \backslash \mathcal{H}$ is the
Riemann sphere $\PP^1$ with four, three, or two points removed.
\hfill $ \diamondsuit $
\end{example}

\begin{example} 
\label{ex:hardyweinberg}
\rm Consider the case $d=2$ of quadrics.
A general quadric has ML degree $6$.
The {\em Hardy-Weinberg curve}, which plays
a fundamental role in population genetics, is given~by
$$ 
f(p_0,p_1,p_2) \,\, = \,\,
{\rm det}
\begin{pmatrix}
2p_0 & p_1 \\
 p_1 & 2p_2
\end{pmatrix} \,\, = \,\, 4 p_0p_2  - p_1^2.
$$
The curve has only three points on the distinguished arrangement:
$$ X \cap \mathcal{H} \,\, = \,\,
\bigl\{ \,(1:0:0), \, (0:0:1), \,  (1:-2:1) \,\bigr\} . $$
Hence the ML degree of the Hardy-Weinberg curve equals $1$.
This means that the maximum likelihood estimate (MLE) is
a rational function of the data. Explicitly, the MLE equals
\begin{equation} \label{eq:hardyweinbergMLE}
( \hat p_0 , \hat p_1, \hat p_2) \,\, = \,\,  \frac{1}{ 4 (u_0+u_1+u_2)^2} \bigl( \,
(2u_0+u_1)^2\, , \,2(2u_0+u_1)(u_1+2u_2)\,, \,(u_1+2u_2)^2\,\bigr).
\end{equation}

In applications, the Hardy-Weinberg curve arises via its parametric representation
\begin{equation}
\label{eq:HW}
\begin{matrix} p_0(s) &=& s^2 \\
                               p_1(s) & = & 2s(1-s) \\
                                p_2(s) & = & (1-s)^2 \end{matrix} 
\end{equation}                                
Here the parameter $s$ is the probability that a biased
coin lands on tails. If we toss that same biased coin twice,
then the above formulas represent the following probabilities:
$$
\begin{matrix}
p_0(s) &  = & \hbox{ probability of 0 heads} \\
p_1(s) & =  & \hbox{ probability of 1 head} \\
p_2(s) & = & \hbox{ probability of 2 heads} \\
\end{matrix}
$$
Suppose now that the experiment of tossing the coin twice
is repeated $N$ times. We record the following counts,
where  $N = u_0+u_1+u_2$ is the sample size of our repeated experiment:
$$
\begin{matrix}
u_0 &=& \hbox{ number of times 0 heads were observed} \\
u_1 &= & \hbox{  number of times 1 head was observed} \\
u_2 &=& \hbox{ number of times 2 heads were observed}
\end{matrix} \qquad \qquad
$$
The MLE problem is to 
estimate the unknown parameter $s$ by maximizing
$$ \ell_{u_0,u_1,u_2}\,\, = \,\, p_0(s)^{u_0} p_1(s)^{u_1} p_2(s)^{u_2} \,\,=\,\,
2^{u_1} s^{2u_0+u_1} (1-s)^{u_1+2u_2}.
$$
The unique solution to this optimization problem is
 $$ \hat s \quad = \quad \frac{2u_0+u_1}{2u_0+2u_1+2u_2} .$$
 Substituting this expression into (\ref{eq:HW}) gives the estimator
 $\bigl(p_0(\hat s), p_1(\hat s), p_2(\hat s)\bigr) $
 for the three probabilities in our model. The
 resulting rational function  coincides with 
(\ref{eq:hardyweinbergMLE}). \hfill $\diamondsuit$
\end{example}

The ML degree is also defined when the given curve $X \subset \mathbb{P}^2$ is not smooth, but it counts critical points of $\ell_u$ 
only in the regular locus of $X$. Here is an example to illustrate this.

\begin{example} \label{ex:nodecusp} \rm
A general cubic curve $X$ in $\PP^2$ has ML degree $12$.
Suppose now that $X$ is a cubic which meets $\mathcal{H}$ transversally
but has one isolated singular point in $\PP^2 \backslash \mathcal{H}$.
If the singular point is a {\em node} then the ML degree of $X$ is $10$,
and if the singular point is a {\em cusp} then the ML degree of $X$ is $9$.
The ML degrees are found by saturating the equations in
(\ref{eq:planefJ}) with respect to the homogenous ideal of the singular point.
\hfill $\diamondsuit$ 
\end{example}

\smallskip

Moving beyond likelihood geometry in the plane,
we shall introduce our objects in any dimension.
We fix the complex projective space $\PP^n$ with  coordinates $p_0,p_1,\ldots,p_n$,
representing probabilities. We summarize the observed data  in a vector
 $u = (u_0,u_1,\ldots,u_n) \in \mathbb{N}^{n+1}$,
 where $u_i$ is the number of samples in state $i$. The likelihood function
  on $\PP^n$ given by $u$ equals
$$ \ell_u \quad = \quad 
\frac{ p_0^{u_0} p_1^{u_1} \cdots p_n^{u_n} }{(p_0+p_1+ \cdots +p_n)^{u_0+u_1+\cdots +u_n}} .$$
 The unique critical point of this rational function on $\PP^n$ is the data point itself:
$$ (u_0:u_1: \cdots : u_n ). $$
Moreover, this point is the global maximum of the likelihood function 
$\ell_u$ on the probability simplex 
$\Delta_n$. Throughout, we identify $\Delta_n$ with the set of all
positive real points in $\PP^n$.

The linear forms in $\ell_u$ define an arrangement
$\mathcal{H}$ of $n+2$ distinguished hyperplanes
in $\PP^n$. The differential of the logarithm of the likelihood function 
is the vector of rational functions
\begin{equation}
\label{eq:nablalog}
{\rm dlog}(\ell_u) \,=\,
\bigl(
\frac{u_0}{p_0},
\frac{u_1}{p_1},\ldots,
\frac{u_n}{p_n} \bigr) 
 - \frac{u_+}{p_+}\cdot (1,1,\ldots,1).
\end{equation}
Here ${p_+} = \sum_{i=0}^n p_i$ and
 ${u_+} = \sum_{i=0}^n u_i$. The vector (\ref{eq:nablalog})
 represents a section of the sheaf 
of differential $1$-forms on $\PP^n$ that have
 logarithmic singularities along $\mathcal{H}$.
This sheaf is denoted
$$ \Omega_{\PP^n}^1({\rm log}(\mathcal{H})). $$

Our aim is to study the restriction of $\ell_u$ to
 a closed subvariety $X  \subseteq \PP^n$.
We will assume that $X$ is
defined over the real numbers, irreducible, and not contained in $\mathcal{H}$.
Let $X_{\rm sing}$ denote the singular locus of $X$, and
$\,X_{\rm reg}$ denote $X \backslash X_{\rm sing}$.
When $X$ serves as a statistical model, the goal is to maximize
the rational function $\ell_u$ on the semialgebraic set $X \cap \Delta_n$. To solve this problem
algebraically, we determine all critical points 
of the log-likelihood function ${\rm log}(\ell_u)$ on
the complex variety  $X$. 
Here we must exclude points that are
singular or lie in $\mathcal{H}$.

\begin{definition} \label{def:MLD_and_LC} \rm
The {\em maximum likelihood degree} of $X$ is the number of complex critical points
of the function $\ell_u$ on $X_{\rm reg}\backslash \mathcal{H}$,
for generic data $u$.
The  {\em likelihood correspondence} $\mathcal{L}_X$ is the
universal family of these critical points. To be precise, $\mathcal{L}_X$ is the
 closure
in $\PP^n \times \PP^n$ of
$$ \bigl\{ (p,u)\,:\,
p \in X_{\rm reg} \backslash \mathcal{H} \,\,\hbox{and} \,\,\,
{\rm dlog}(\ell_u) \,\, \hbox{vanishes at $p$} \bigr\} .$$
\end{definition}

\smallskip

We sometimes write $\,\PP^n_p \times \PP^n_u\,$ for $\,\PP^n \times \PP^n\,$
to highlight that the first factor is the probability space, with coordinates $p$,
while the second factor is the data space, with coordinates $u$.
The first part of the following result appears in \cite[\S 2]{Huh1}.
A precursor was \cite[Proposition 3]{HKS}.

\begin{theorem} \label{thm:finite-to-one}
The likelihood correspondence
$\mathcal{L}_X$ of any irreducible
subvariety $X $ in   $ \PP^n_p$ is an irreducible variety of dimension $n$ in 
the product $\,\PP^n_p \times \PP^n_u$.
The map ${\rm pr}_1:\mathcal{L}_X \rightarrow \PP^n_p$ 
is a projective bundle over $X_{\rm reg}\backslash \mathcal{H}$, and
the map ${\rm pr}_2:\mathcal{L}_X \rightarrow \PP^n_u$ is 
generically finite-to-one.
\end{theorem}

See Section \ref{Proofs} for a proof. The degree of the map ${\rm pr}_2:\mathcal{L}_X \rightarrow \PP^n_u$ to data space is the ML degree of $X$.
This number has a topological interpretation as an Euler characteristic,
provided suitable assumptions on $X$ are being made.
The relationship between the homology of a manifold
and critical points of a suitable function on it is the topic
of Morse theory.

The study of ML degrees was started in \cite[\S 2]{CHKS} by developing
 the connection to  the sheaf
$\,\Omega^1_X({\rm log}(\mathcal{H}))$ of differential $1$-forms on $X$
with logarithmic poles along $\mathcal{H}$.
It was shown in \cite[Theorem 20]{CHKS} that
the ML degree of $X$ equals the signed topological Euler characteristic 
\[
(-1)^{\dim X} \cdot \chi(X \backslash \mathcal{H}),
\] 
provided
$X$ is smooth
and the intersection $\mathcal{H} \cap X$ defines a normal crossing divisor in 
$X \subseteq \PP^n$.
A major drawback of that early result was that the hypotheses are so restrictive that they
essentially never hold for varieties $X$ that arise from statistical models used in practice.
From a theoretical point view, this issue can be addressed by passing to a
resolution of singularities.
However, in spite of existing algorithms for resolution in characteristic zero,
these algorithms do not scale to problems of the sizes of interest in
algebraic statistics. Thus, whatever
computations we wish to do should not be based on resolution of singularities.

The following result due to \cite{Huh1} gives the same topological 
interpretation of the ML degree. The hypotheses here
are much more realistic and inclusive than those in
\cite[Theorem 20]{CHKS}.

\begin{theorem} \label{thm:eulerchar}
If the very affine variety $X \backslash \mathcal{H}$ is smooth
 of dimension $d$,
then the ML degree of $X$ equals the signed topological
Euler characteristic of $(-1)^d \cdot   \chi(X \backslash \mathcal{H})$.
\end{theorem}

The term {\em very affine variety} refers to a closed subvariety of some algebraic torus
$(\mathbb{C}^*)^m$.
Our ambient space $\mathbb{P}^n \backslash \sH$ is a very affine variety because it has a closed embedding
\[
\mathbb{P}^n \backslash \sH \longrightarrow (\mathbb{C}^*)^{n+1},
 \qquad (p_0:\cdots:p_n) \longmapsto \big(\frac{p_0}{p_+},\ldots,\frac{p_n}{p_+}\big).
\]
The study of such varieties is foundational for {\em tropical geometry}.
The special case when $X \backslash \sH$ is a Riemann surface with $a$ punctures,
arising from a curve in $ \PP^2$,
was seen in  Theorem \ref{thm:RiemannSurface}. 
We remark that Theorem \ref{thm:eulerchar} can be deduced from works of Gabber-Loeser \cite{Gabber-Loeser} and Franecki-Kapranov \cite{Franecki-Kapranov} on perverse sheaves on algebraic tori.

The smoothness hypothesis is essential for  Theorem \ref{thm:eulerchar} to hold.
If $X$ is singular then, generally, neither 
$X \backslash \mathcal{H}$ nor $ X_{\rm reg} \backslash \mathcal{H}$
has its signed Euler characteristic equal to the ML degree of $X$. 
Varieties $X$ that demonstrate this are the two singular cubic curves
in Example \ref{ex:nodecusp}.

\begin{conjecture} \label{conj:eulerpositive}
For any projective variety $X \subseteq \mathbb{P}^n$ of dimension $d$, not contained in $\sH$,
\[
(-1)^d \cdot \chi(X \backslash \sH) \,\,\ge \,\, {\rm MLdegree}\,(X).
\]
In particular, the signed topological Euler characteristic $(-1)^d \cdot
 \chi(X \backslash \sH)$ is nonnegative.
\end{conjecture}

Analogous conjectures can be made
in the slightly more general setting of \cite{Huh1}. In particular, we conjecture that the inequality
\[
(-1)^d \cdot \chi(V) \,\ge \, 0
\]
holds for any closed $d$-dimensional subvariety $V \subseteq (\mathbb{C}^*)^{m}$.

\begin{remark} \rm
We saw in Example \ref{ex:linesinplane}   that
the ML degree of a projective variety $X$
can be $0$. In all situations
of statistical interest, the variety $X \subset \PP^n$ 
intersects the open simplex $\Delta_n$ in a subset that is Zariski dense in $X$.
If that intersection is smooth  then  $ {\rm MLdegree}(X) \geq 1$. In fact,
arguing as in \cite[Proposition 11]{CHKS}, it can be shown that
for smooth $X$,
\[
{\rm MLdegree}\,(X) \,\,\ge \,\,\#(\text{bounded regions of $X_{\mathbb{R}} \backslash \mathcal{H}$}).
\]
Here a {\em bounded region} is a connected component of
the semialgebraic set $X_{\mathbb{R}} \backslash \mathcal{H}$
whose classical closure is disjoint from 
the distinguished hyperplane $V(p_+)$  in $\PP^n$.

If $X$ is singular then the number of bounded regions of
 $X_{\mathbb{R}} \backslash \mathcal{H}$
can exceed ${\rm MLdegree}\,(X)$.
For instance, let $X\subset \PP^2$ be the cuspidal cubic curve defined by
\[
(p_0{+}p_1{+}p_2)(7 p_0 {-} 9 p_1 {-} 2 p_2)^2 =
(3p_0{+}5p_1{+}4p_2)^3.
\]
The real part
 $X_{\mathbb{R}} \backslash \mathcal{H}$
consists of $8$ bounded and $2$ unbounded regions, but the ML degree of $X$ is $7$.
The bounded region that contains the cusp $(13:17:-31)$  
has no other critical points for $\ell_u$.
\hfill $ \diamondsuit $
\end{remark}

In what follows we present instances that illustrate the
computation of the ML degree. We begin
with the case of
{\em generic complete intersections}.
Suppose that $X \subset \mathbb{P}^n$ is a complete intersection defined by
$r$ generic homogeneous polynomials $g_1,\ldots,g_r$ of degrees $d_1,d_2,\ldots,d_r$.

\begin{theorem} \label{thm:genericMLdegree}
The ML degree of $X$ equals $D d_1 d_2 \cdots d_r $, where
\begin{equation}
\label{eq:upperboundformula}
 D \,\quad = \, 
\sum_{i_1+i_2 + \cdots + i_r \leq n-r}  \!\!\!\!\!
d_1^{i_1} d_2^{i_2} \cdots d_r^{i_r}. 
\end{equation}
\end{theorem}

\begin{proof}
By Bertini's Theorem, the generic complete 
intersection $X$ is smooth in $\mathbb{P}^n$.
All critical points of the likelihood function $\ell_u$ on $X$ 
lie in the dense open subset $X \backslash \mathcal{H}$.
 Consider the following $(r+2) \times (n+1)$-matrix with entries in 
 the polynomial ring $\R[p_0,p_1,\ldots,p_n]$: 
\begin{equation}
\label{eq:uJmatrix}
 \bmatrix u \\ {\tilde J}(p) \endbmatrix \quad = \quad 
\bmatrix 
& u_0 & u_1 & \cdots & u_n \\ 
& p_0 & p_1 & \cdots & p_n \smallskip \\ 
& 
p_0 \frac{\partial g_1}{\partial p_0} & 
p_1 \frac{\partial g_1}{\partial p_1} & \cdots & 
p_n \frac{\partial g_1}{\partial p_n} \medskip \\ 
& 
p_0 \frac{\partial g_2}{\partial p_0} & 
p_1 \frac{\partial g_2}{\partial p_1} & \cdots & 
p_n \frac{\partial g_2}{\partial p_n} \\ 
&  \vdots & \vdots & \ddots & \vdots & \smallskip \\ 
& 
p_0 \frac{\partial g_r}{\partial p_0} & 
p_1 \frac{\partial g_r}{\partial p_1} & \cdots & 
p_n \frac{\partial g_r}{\partial p_n} 
\endbmatrix . 
\end{equation}
 Let $Y$ denote the determinantal variety in $\mathbb{P}^n$ given by the vanishing of its 
$(r+2) \times (r+2)$ minors.  
The codimension of $Y$ is at most 
$n-r$, which is a general upper bound for ideals of maximal 
minors, and hence the dimension of $Y$ is at least $r$. 
Our genericity assumptions ensure that 
the matrix $\tilde J(p)$ has maximal row rank $r+1$ for all $p \in X$. 
Hence a point $p \in X$ lies in $Y$ 
if and only if the vector $u$ is in the row span of $\tilde J(p)$. 
Moreover, by Theorem \ref{thm:finite-to-one}, 
$$ (X_{\rm reg} \backslash \mathcal{H}) \,\cap \, Y \quad = \quad X \, \cap \, Y  $$ 
is a finite subset of $\mathbb{P}^n$, and its cardinality is the 
desired ML degree of $X$.

 Since $X$ has dimension  
$n-r$, we conclude that $Y$ has the maximum possible codimension,  
namely $n-r$,  and that the 
intersection of $X$ with the determinantal variety $Y$ is proper. 
We note that $Y$ is  Cohen-Macaulay, since $Y$ has  
maximal codimension $n-r$, and  
ideals of minors  of generic matrices are Cohen-Macaulay. 
B\'ezout's Theorem implies 
$$ 
{\rm ML degree}(X) \quad = \quad {\rm degree}(X) \cdot {\rm 
degree}(Y) \quad = \quad d_1 \cdots d_r \cdot {\rm degree}(Y). $$ 
The degree of the determinantal variety $Y$  equals 
the degree of the determinantal variety given by generic 
forms of the same row degrees.   By the
Thom-Porteous-Giambelli formula, this degree is 
the complete homogeneous symmetric function of degree $\,{\rm 
codim}(Y) = n-r\,$ evaluated at the row degrees of the matrix. 
Here, the row degrees are $\,0,1,d_1, \ldots, d_r $, and the value 
of that symmetric function is precisely $D$. 
We conclude that ${\rm degree}(Y) = D$. 
Hence the ML degree 
of the  generic complete intersection $\,X = \mathcal{V}( g_1,\ldots,g_r ) \,$ equals 
$\, D \cdot d_1 d_2 \cdots d_n $. 
\end{proof}

\begin{example} [$r=1$]  \label{ex:generichyper} \rm
A generic hypersurface of degree $d$ in $\mathbb{P}^n$ has ML degree
$$ d \cdot D \,\, = \,\, d + d^2 + d^3 + \cdots + d^n. $$
\end{example}

\begin{example}[$r=2,n = 3$] \label{eq:twosurfaces} \rm
A  space curve  that is the generic intersection
of two surfaces of degree $d$ and $e$ in $\mathbb{P}^3$ has ML degree
$\, de + d^2e + d e^2 $. \hfill $\diamondsuit$
\end{example}

\begin{remark} \rm
It was shown in \cite[Theorem 5]{HKS} that 
(\ref{eq:upperboundformula}) is an upper bound for the ML degree of any
variety $X$ {\em of codimension $r$} that is defined by polynomials of degree $d_1,\ldots,d_r$.
In fact, the same is true under the weaker hypothesis that $X$ is cut out by
polynomials of degrees $d_1 \geq \cdots \geq d_r \geq d_{r+1} \geq \cdots \geq d_s$,
so $X$ need not be a complete intersection.
However, the hypothesis ${\rm codim}(X) = r$ is essential in order
for ${\rm MLdegree}(X) \leq $ (\ref{eq:upperboundformula})  to hold.
That codimension hypothesis was forgotten when this
upper bound was cited in \cite[Theorem 2.2.6]{LiAS}
and in  \cite[Theorem 3.31]{PS}.
Hence these two book references are not
correct as stated. 

Here is a  simple counterexample. Let $n=3$ and $d_1=d_2 = d_3 = 2$.
Then the bound (\ref{eq:upperboundformula}) is the B\'ezout number $8$,
and this is also the correct ML degree for a general complete intersection of three quadrics in $\PP^3$.
Now let $X$ be a general rational normal curve in $\PP^3$.
The curve $X$ is  defined by three quadrics, namely, the
$2 \times 2$-minors of a $2 \times 3$-matrix filled
with general linear forms in
$p_0,p_1,p_2,p_3$. Since $X$ is a Riemann sphere with $15$ punctures,
Theorem \ref{thm:eulerchar} tells us that ${\rm MLdegree}(X) = 13$,
and this exceeds the bound of $8$.
\hfill $\diamondsuit $
\end{remark}

We now come to a variety that is ubiquitous in statistics, namely the model of
{\em independence} for two binary random variables \cite[\S 1.1]{LiAS}.
This model is represented by Segre's quadric surface $X$ in $\mathbb{P}^3$.
By this we mean the surface  defined by the $2 \times 2$-determinant:
$$ X \, = \, V(p_{00} p_{11} - p_{01} p_{10} ) \,\,\subset \,\, \PP^3. $$
The surface $X$ is isomorphic to $\PP^1 \times \PP^1$, so it is smooth,
and we can apply Theorem \ref{thm:eulerchar} to find the ML degree.
In other words, we seek  to determine the Euler characteristic of the 
open complex surface $\, X \backslash \mathcal{H}\,$ where
\[
\mathcal{H} = \bigl\{\, p \in \PP^3:
p_{00} p_{01} p_{10} p_{11} (p_{00} {+} p_{01} {+} p_{10} {+} p_{11}) = 0 \bigr\}.
\]
To this end, we write $X = \PP^1 \times \PP^1$
with coordinates $\bigl((x_0:x_1),(y_0:y_1) \bigr)$.
Our surface is parametrized by $p_{ij} = x_i y_j$, and hence
$$
\begin{matrix}
X \backslash \mathcal{H} &= &
\big(\PP^1 \times \PP^1\big) \backslash \big\{ x_0x_1y_0y_1(x_0+x_1)(y_0+y_1) = 0\big\}
\qquad \qquad \smallskip
\\ &= & 
\big(\PP^1 \backslash \{x_0x_1(x_0+x_1) = 0\}\big) \,\times\,
\big(\PP^1 \backslash \{y_0y_1(y_0+y_1) = 0\}\big) 
\smallskip
\\ & = &
\big(\hbox{$2$-sphere} \backslash \{\hbox{three points}\}\big) \,\times\,
\big(\hbox{$2$-sphere} \backslash \{\hbox{three points}\}\big) .
\end{matrix}
$$
Since the Euler characteristic is additive and multiplicative, 
$$ \chi(X \backslash \mathcal{H}) \,\,=\,\, (-1) \cdot (-1)\,\, =\,\, 1. $$
This means that the map $u \mapsto \hat p$ from
the data to the MLE is a rational function in each coordinate.
The following ``word problem for freshmen'' is
aimed at finding that  function.

\begin{example} \label{ex:biologist} \rm
Do this exercise:
A biologist friend of yours wishes to test whether two binary random variables
are independent. She collects data and records the
matrix of counts 
$$\,u  \,\,=\,\, \begin{pmatrix} u_{00} & u_{01} \\
u_{10} & u_{11} \end{pmatrix} .$$
How to ascertain whether
 $u$ lies close to the independence model
 $$ \qquad X  \,\,=\,\, V(p_{00} p_{11} - p_{01} p_{10}) \,? $$
A statistician who recently started working in her lab explains that, as the first step in the
analysis of her data,  the biologist should calculate the
 maximum likelihood estimate (MLE)
$$ \hat p \, = \, \begin{pmatrix} \hat p_{00} & \hat p_{01} \\
\hat p_{10} & \hat p_{11} \end{pmatrix} . $$
Can you help your friend by supplying the formula for
$\hat p$ as a rational function in $u$?

\smallskip

The solution to this word problem is as follows.
 The MLE is the rank $1$ matrix
 \begin{equation}
 \label{eq:MLE22}
 \hat p \, = \,\frac{1}{(u_{++})^2}
          \begin{pmatrix} u_{0+} \\ u_{1+} \end{pmatrix}
\cdot \begin{pmatrix} u_{+0} & u_{+1} \end{pmatrix}. 
\end{equation}
We illustrate the concepts introduced above
by deriving this well-known formula.
The likelihood correspondence   $\mathcal{L}_X$ 
of $\,X = V(p_{00} p_{11} - p_{01} p_{10})\,$
is the subvariety of $X \times \mathbb{P}^3$ defined~by 
\begin{equation}
\label{eq:MLE22b}
 U \cdot ( p_{00},p_{01},p_{10},p_{11})^T  \,\, = \,\,0 , 
 \end{equation}
where $U$ is the matrix 
$$
U \,\, = \,\,
\begin{pmatrix}
        0   & -u_{10} - u_{11} &      0   & u_{00}+u_{01} \\
   u_{11}+u_{01} &  -u_{00}-u_{10} &     0  &    0    \\
    u_{11}+u_{10} &    0  & -u_{01}-u_{00} &    0    \\
        0   &   0  & -u_{01}-u_{11} & u_{00}+u_{10} 
 \end{pmatrix}.
$$
 We urge the reader to derive
(\ref{eq:MLE22b}) 
from Definition \ref{def:MLD_and_LC}
using a computer algebra system.

Note that the determinant of $U$ vanishes identically.
In fact, for generic $u_{ij}$, the matrix $U$ has rank $3$,
so its kernel is spanned by a single vector. The coordinates
of that vector are given by Cramer's rule, and we find them
to be equal to the rational functions in (\ref{eq:MLE22}).

The locus where the function $u \mapsto \hat p$ is undefined
consists of those $u$  where the matrix rank of $U$ drops below $3$.
A computation shows that the rank of $U$ drops
 to $2$ on the variety 
$$V(u_{00}+u_{10},u_{01}+u_{11}) \,\cup \,
V(u_{00}+u_{01},u_{10}+u_{11}), $$
and it drops to $0$ on  the point $\,V(u_{00}+u_{01},u_{10}+u_{11},u_{01}+u_{11})$.  
In particular, the likelihood function 
$\ell_u$ given by that point $u$
 has infinitely many critical points
in the quadric  $X$.
\hfill $\diamondsuit$
\end{example}

We note that all coefficients of the linear forms 
that define the exceptional loci in $\PP^3_u$ for
the independence model are positive. This 
means that data points $u$ with all coordinates positive
can never be exceptional.
We will prove in Section \ref{Proofs} that this usually holds.
Let ${\rm pr}_1:\sL_X \to \mathbb{P}^n_p$ and ${\rm pr}_2: \sL_X \to \mathbb{P}^n_u$ be the projections from the likelihood correspondence to
  $p$-space and  $u$-space respectively.
We are interested in the fibers of  ${\rm pr}_2$ over
positive points~$u$.

\begin{theorem}\label{PositiveData}
Let $u \in \mathbb{R}^{n+1}_{> 0}$, and
let $X \subset \PP^n$ be an
irreducible variety such that no
singular points of any intersection
 $X \cap \{p_i = 0\}$ lies in the hyperplane at infinity $\{p_+ = 0\}$. Then
\begin{enumerate}
\item the likelihood function $\ell_u$ on $X$ has only finitely many critical 
points in  $X_{\rm reg} \backslash \mathcal{H}$;
\item if the fiber ${\rm pr}_2^{-1}(u)$ is contained in $X_{\rm reg}$, then its length equals the ML degree of $X$.
\end{enumerate}
\end{theorem}

The hypothesis concerning ``no singular point'' will be satisfied for
essentially all statistical models of interest.
Here is an example which shows that this hypothesis is necessary.

\begin{example}
\label{ex:may31shocker} \rm
We consider the smooth cubic curve $X$ in $\PP^2$ that is defined by
$$ f\, \,\,= \,\,\, (p_0+p_1+p_2)^3 \,+\, p_0p_1p_2 .$$
The  ML degree of the curve $X$ is $3$.
Each intersection $X \cap \{p_i = 0\}$ is a triple point that lies
on the line at infinity $\{p_+ = 0\}$.
  The fiber ${\rm pr}_2^{-1}(u)$
of the likelihood fibration over the positive point $u = (1:1:1)$ is
the entire curve $X$.
\hfill $\diamondsuit $
\end{example}

If $u$ is not positive
in Theorem \ref{PositiveData},
 then the fiber of ${\rm pr}_2$ over $u$ may have positive dimension. We saw an instance of this at the end of Example \ref{ex:biologist}.
Such {\em resonance loci} have been studied extensively 
 when $X$ is a linear subspace of $\mathbb{P}^n$. See \cite{CDFV} and references therein.

The following cautionary example shows that
the length of the scheme-theoretic fiber of $\sL_X \to \mathbb{P}^n_u$ over
special points $u$ in the open simplex $\Delta_n$
may exceed the ML degree of~$X$.

\begin{example} \label{ex:cautionary}  \rm
Let $X$ be the curve in $\mathbb{P}^2$ defined by the ternary cubic
\[ f\,\, = \,\, p_2(p_1-p_2)^2+(p_0-p_2)^3 .
\]
This curve intersects $\sH$ in $8$ points, has ML degree $5$, and has a cuspidal singularity at
\[
P\,\,:=\,\,(1:1:1).
\]
The prime ideal in $\R[p_0,p_1,p_2,u_0,u_1,u_2]$ 
for the likelihood correspondence $\mathcal{L}_X$  is
minimally generated by five polynomials, having
degrees $(3,0), (2,2), (3,1), (3,1), (3,1)$. They are
 obtained by saturating the
two equations in (\ref{eq:planefJ}) with respect to 
$\,\langle p_0 p_2 \rangle \cap \langle p_0-p_1,p_2-p_1 \rangle$.

The scheme-theoretic fiber of ${\rm pr}_1$ over a general point of $X$ is a reduced line in the $u$-plane, while the fiber of ${\rm pr}_1$ over $P$ is the double line
\[
L:= \big\{\, (u_0:u_1:u_2) \in \mathbb{P}^2 : (2u_0-u_1-u_2)^2=0\, \big\} .
\]
The reader is invited to verify the following assertions using a computer algebra system:
\begin{enumerate}
\item[(a)] If $u$ is a general point of $\mathbb{P}^2_u$, then $\text{pr}_2^{-1}(u)$ 
consists of $5$ reduced points in $X_{\rm reg}\backslash \sH$. 
\item[(b)]  If $u$ is a general point on the line $L$, then 
the locus of critical points $\text{pr}_2^{-1}(u)$ consists 
of $4$ reduced points in $X_{\rm reg} \backslash \sH$ and the reduced point $P$.
\item[(c)] If $u$ is the point $(1:1:1) \in L$, then
$\text{pr}_2^{-1}(u)$ is a zero-dimensional scheme of length $6$. This scheme consists of $3$ reduced points in $X_{\rm reg} \backslash \sH$ and $P$ counted with multiplicity $3$.
\end{enumerate}
In particular, the fiber in (c) is not algebraically equivalent to the general fiber
(a). This example illustrates one of the difficulties classical geometers had to face when formulating the ``principle of conservation of numbers''. See \cite[Chapter 10]{Fulton} for a modern treatment.
\hfill $\diamondsuit$
\end{example}

\smallskip

It is instructive to examine  classical varieties from projective geometry from the
likelihood perspective. For instance, we may study the Grassmannian in its Pl\"ucker 
embedding. Grassmannians are a nice test case because they are
smooth, so that Theorem \ref{thm:eulerchar} applies.

\begin{example}\label{ex:Grassmannian} \rm
Let $X = G(2,4)$ denote the Grassmannian of lines in $\PP^3$.
In its Pl\"ucker embedding in $\mathbb{P}^5$, this Grassmannian is the quadric
hypersurface defined by
\begin{equation}
\label{eq:pluckerquadric}
 p_{12} p_{34} - p_{13} p_{24} + p_{14} p_{23} \,\,= \,\,0 . 
\end{equation}
As in (\ref{eq:uJmatrix}),
the critical equations for the likelihood function $\ell_u$ are 
the $3 \times 3$-minors of
\begin{equation}
\label{eq:G24jacobian}
 \begin{bmatrix}
       u_{12} &   u_{13} &    u_{14} &   u_{23} &   u_{24} &   u_{34} \\
       p_{12} &    p_{13} &    p_{14} &   p_{23} &  p_{24} &   p_{34} \\
\,     p_{12} p_{34} &   - p_{13} p_{24} &   p_{14} p_{23} &  p_{14} p_{23} &
         - p_{13} p_{24} &  p_{12} p_{34} \,
\end{bmatrix} .
\end{equation}
By Theorem  \ref{thm:finite-to-one},
the likelihood correspondence $\mathcal{L}_X$ is a
five-dimensional subvariety of $\mathbb{P}^5 \times \mathbb{P}^5$.
The cohomology class of this subvariety can be represented by the bidegree
of its ideal:
\begin{equation}
\label{eq:GrassmannBidegree}
B_X(p,u) \quad = \quad
{\bf 4} p^5 \,+\, 6 p^4 u \,+\, 6 p^3 u^2 \,+\, 6 p^2 u^3 \,+\, 2pu^4 .
\end{equation}
This is the {\em multidegree}, in the sense of
\cite[\S 8.5]{MS}, of $\mathcal{L}_X$ with respect to the
natural $\mathbb{Z}^2$-grading on the polynomial ring $\mathbb{R}[p,u]$.
We can use \cite[Proposition 8.49]{MS} to
compute the bidegree from the prime ideal of $\mathcal{L}_X$.
Its leading coefficient $4$ is the ML degree of $X$.
Its trailing coefficient $2$ is the degree of $X$.
The polynomials $B_X(p,u)$ will be studied in Section~3.

The prime ideal of $\mathcal{L}_X$ is computed from the equations 
in (\ref{eq:pluckerquadric}) and (\ref{eq:G24jacobian})
by saturation with respect to  $\mathcal{H}$.
It is minimally generated by the following eight polynomials in $\mathbb{R}[p,u]$:
\begin{itemize}
\item[(a)] one polynomial of  degree $(2,0)$, namely the Pl\"ucker quadric,
\item[(b)] six polynomials of degree $(1,1)$, given by $2 \times 2$-minors of
$$ \begin{pmatrix}
p_{12}-p_{34} & p_{13}-p_{24} & p_{14}-p_{23} \\
u_{12}-u_{34} &  u_{13}-u_{24} & u_{14}-u_{23} 
\end{pmatrix}  
\qquad \hbox{and}
$$
$$ \!\!\! \begin{pmatrix}
 p_{12}{+}p_{13}{+}p_{23} & p_{12}{+}p_{14}{+}p_{24} & p_{13}{+}p_{14}{+}p_{34} & p_{23}{+}p_{24}{+}p_{34} \\
 u_{12}{+}u_{13}{+}u_{23} & u_{12}{+}u_{14}{+}u_{24} & u_{13}{+}u_{14}{+}u_{34} & u_{23}{+}u_{24}{+}u_{34} 
 \end{pmatrix},
 $$
 \item[(c)] one polynomial of  degree $(2,1)$, for instance
$$ 
\begin{matrix}
  2 u_{24} p_{12}p_{34}
+2 u_{34} p_{13} p_{24}
+(u_{23}+u_{24}+u_{34}) p_{14} p_{24}\\
 -(u_{13}+u_{14}+u_{34})p_{24}^2
-(u_{12}+2 u_{13}+u_{14}-u_{24}) p_{24} p_{34}.
\end{matrix}
      $$
\end{itemize}
For a fixed positive data vector $u > 0$, these six polynomials in (b)
reduce to three linear equations, and these cut out a plane $\PP^2 $ inside
$ \PP^5$. To find the four critical points of $\ell_u$ on $X = G(2,4)$,
  we must then intersect the two conics (a) and (c) in that plane $\PP^2$.

The ML degree of the Grassmannian $G(r,m)$ in
$\PP^{\binom{m}{r}-1}$ is the signed Euler characteristic
of the manifold $\,G(r,m) \backslash \mathcal{H}\,$ obtained
by removing $\binom{m}{r}+1$ distinguished hyperplane sections.
It would be very interesting to find a general formula for this ML
degree. At present, we only know that
the ML degree of $G(2,5)$ is $26$, and that
the ML degree of $G(2,6)$ is $156$.
By Theorem \ref{thm:eulerchar}, these 
numbers give the Euler characteristic of 
$G(2,m) \backslash \mathcal{H}$ for $m \leq 6$.
 \hfill $\diamondsuit$
\end{example}

We end this lecture with a discussion of the delightful case
when $X$ is a linear subspace of $\PP^n$, and the open variety
 $X \backslash \mathcal{H}$ is the complement
of a {\em hyperplane arrangement}.
In this context,  following Varchenko \cite{Varchenko}, the  likelihood function
$\ell_u$ is known as as the {\em master function},
and  the statement of Theorem \ref{thm:eulerchar} 
was  first proved by Orlik and Terao in \cite{Orlik-Terao}.
We assume that
$X$ has dimension $d$, is defined over $\mathbb{R}$,
and does not contain the vector ${\bf 1} = (1,1,\ldots,1)$.
We can regard $X$ as a
$(d+1)$-dimensional linear subspace of $\R^{n+1}$.
The orthogonal complement $X^\perp$ with respect to the
standard dot product is a linear space of dimension $n-d$ in $\R^{n+1}$.
The linear space $X^{\perp} + {\bf 1}$ spanned by $X^\perp$ and the vector ${\bf 1}$
 has dimension $n-d+1$ in $\R^{n+1}$,
and hence can be viewed as subspace of codimension $d$ in $\PP^n_u$.
In our next formula,
 the operation $\star$ is the Hadamard product or coordinatewise product.

\begin{proposition} \label{prop:linLC}
The likelihood correspondence $\mathcal{L}_X$ in $\PP^n \times \PP^n$ is defined by
\begin{equation}
\label{eq:linearlikelihood}
 p \in X \quad \hbox{and} \quad
u \,\in\, p \star (X^\perp + {\bf 1}).
\end{equation}
The prime ideal of $\mathcal{L}_X$ is obtained from these
constraints by saturation with respect to $\mathcal{H}$.
\end{proposition}

\begin{proof}
If all $p_i$ are non-zero then
$\,u \,\in\, p \star (X^\perp + {\bf 1})\,$ says that
\[
u/p := \big(\frac{u_0}{p_0}, \frac{u_1}{p_1}, \ldots, \frac{u_n}{p_n}\big)
\]
lies in the subspace
$X^{\perp} + {\bf 1}$. Equivalently, the vector obtained by
adding a multiple of $(1,1,\ldots,1)$ to $u/p$ is perpendicular to $X$.
We can take that vector to be the differential (\ref{eq:nablalog}).
Hence  (\ref{eq:linearlikelihood})
expresses the condition that
$p$ is a critical point of $\ell_u$ on $X$.
\end{proof}

The intersection $X \cap \mathcal{H}$
is an arrangement of $n+2$ hyperplanes in $X \simeq \PP^d$.
For special choices of the subspace $X$, it may happen that two or more hyperplanes coincide.
Taking $\{p_+ = 0\}$ as the hyperplane at infinity, 
we view   $X \cap \mathcal{H}$  as an
arrangement of $n+1$ hyperplanes in the  affine space $\mathbb{R}^d$.
A region of this arrangement is {\em bounded} if it is disjoint from $\{p_+ = 0\}$.

\begin{theorem}\label{LinearCase}
The ML degree of $X$ is the number of bounded regions of the real affine hyperplane arrangement $X \cap \sH$ in $\mathbb{R}^d$. The bidegree of the likelihood correspondence $\sL_X$ is the 
$h$-polynomial of the broken circuit complex of
the rank $d{+}1\!$ matroid associated with $X \cap \sH$.
\end{theorem}

We need to explain the second assertion. The hyperplane
arrangement  $X \cap \sH$ consists of the intersections
of the $n+2$ hyperplanes in $\sH$ with $X \simeq \PP^d$. We
regard these as hyperplanes through the origin in $\R^{d+1}$. They
define a matroid $M$ of rank $d+1$  on $n+2$ elements.
We identify these elements
with the variables $x_1,x_2,\ldots,x_{n+2}$.
For each circuit $C$ of $M$ let  
$m_C = (\prod_{i \in C} x_i)/x_j$ 
where $j$ is the smallest index such that $x_j \in C$.
The {\em broken circuit complex} of $M$ is the 
simplicial complex with Stanley-Reisner ring
$\, \mathbb{R}[x_1,\ldots,x_{n+2}]/ \langle \,m_C \,: \,C \,\,\hbox{circuit of} \,\, M \,\rangle$.
See \cite[\S 1.1]{MS} for Stanley-Reisner basics.
The Hilbert series of this graded ring has the form
$$
\frac{h_0 + h_1 z + \cdots + h_{d} z^d}{(1-z)^{d+1}}.
$$
What is being claimed in Theorem \ref{LinearCase} is that
the bidegree of $\mathcal{L}_X$ equals
\begin{equation}
\label{eq:linearBidegree}
 B_X(p,u) \quad = \quad
(h_0 u^d + h_1 p u^{d-1}  + h_2 p^2 u^{d-2} + \cdots + h_d p^d) \cdot p^{n-d}
\end{equation}
Equivalently, this is the class of $\sL_X$ in the cohomology ring
\[
H^*(\mathbb{P}^n \times \mathbb{P}^n;\mathbb{Z}) =
\mathbb{Z}[p,u]/\langle p^{n+1}, u^{n+1} \rangle.
\]
There are several (purely combinatorial) definitions of the invariants $h_i$
of the matroid $M$. For instance, they are
coefficients of the following specialization of the
{\em characteristic polynomial}:
\begin{equation}
\label{eq:chih}
\chi_M(q+1)\quad = \quad  q \cdot \Big(h_0 q^d-h_{d-1}q^{d-1}+\cdots+(-1)^{d-1} h_1 q +(-1)^d h_0\Big).
\end{equation}
Theorem \ref{LinearCase} was used in \cite{Huh0} to prove a conjecture of Dawson,
stating that the sequence  $h_0,h_1, \ldots,h_d$ is log-concave,
 when $M$ is representable over a field of characteristic zero.

The first assertion in Theorem \ref{LinearCase} was  proved by Varchenko in \cite{Varchenko}.
For definitions and characterizations of the characteristic polynomial $\chi$, and many 
pointers to matroid basics, we refer to \cite{OTBook}.
 A proof of the second assertion was given by Denham et al. in a slightly different setting \cite[Theorem 1]{Denham-Garrousian-Schulze}. We give a proof in Section \ref{Proofs} following \cite[\S 3]{Huh1}.
The ramification locus of the likelihood fibration
${\rm pr}_2:\mathcal{L}_X \rightarrow \PP^n_u\,$ is known as the 
{\em entropic discriminant} \cite{SSV}.

\begin{example} \rm
Let $d=2$ and $ n=4$, so $X$ is a plane in $\PP^4$, defined
by two linear forms
\begin{equation}
\label{eq:ex24A}
\begin{matrix}
 c_{10} p_0 + c_{11} p_1 + c_{12} p_2 + c_{13} p_3 + c_{14} p_4 & = &  0, \\
 c_{20} p_0 + c_{21} p_1 + c_{22} p_2 + c_{23} p_3 + c_{24} p_4 & = & 0. \\
 \end{matrix}
\end{equation}
Following Theorem  \ref{LinearCase}, we view
$X \cap \mathcal{H} $ as an arrangement of five lines in the affine plane
$$  \{\, p \in X \,: \, p_0 + p_1 + p_2 + p_3 + p_4 \not= 0 \,\} \quad \simeq \,\,\, \mathbb{C}^2. $$
Hence, for generic $c_{ij}$, the ML degree of $X$ is equal to $6$,
the number of bounded regions of this arrangement.
The condition $u \,\in\, p \star (X^\perp + {\bf 1})$ 
in Proposition \ref{prop:linLC} translates into
\begin{equation}
\label{eq:ex24B} {\rm rank} \begin{bmatrix}
u_0 & u_1 & u_2 & u_3 & u_4 \\
p_0 & p_1 & p_2 & p_3 & p_4 \\
 c_{10} p_0 & c_{11} p_1 & c_{12} p_2 & c_{13} p_3 & c_{14} p_4  \\
  c_{20} p_0 & c_{21} p_1 & c_{22} p_2 & c_{23} p_3 & c_{24} p_4 
  \end{bmatrix}
\,  \leq \,3 .  
\end{equation}
The $4 \times 4$-minors of this $4 \times 5$-matrix,
together with the two linear forms defining $X$,
form a system of equations that has six solutions in $\PP^4$, for generic $c_{ij}$.
All solutions have real coordinates. In fact, there is one solution in each bounded region
of $X \backslash \mathcal{H}$.
The likelihood correspondence $\mathcal{L}_X$ is the fourfold in $\PP^4 \times \PP^4$
given by the equations (\ref{eq:ex24A}) and (\ref{eq:ex24B}). 

We now illustrate the second statement  in Theorem \ref{LinearCase}.
Suppose that the  real numbers $c_{ij}$ are generic, so $M$ is the uniform
matroid of rank three on six elements. The Stanley-Reisner ring of the
 broken circuit complex of $M$ equals
 $$ \R [x_1,x_2,x_3,x_4,x_5,x_6]/\langle 
x_2 x_3x_4,x_2 x_3x_5,x_2 x_3x_6,
\ldots, x_4 x_5 x_6
\rangle. $$
The Hilbert series of this graded algebra is 
$$ \frac{h_0 + h_1 z + h_2 z^2}{(1-z)^3} \quad = \quad
\frac{1 + 3 z + 6 z^2}{(1-z)^3}. $$
We conclude that  the bidegree 
(\ref{eq:linearBidegree})
of the likelihood correspondence $\mathcal{L}_X$ equals
$$ B_X(p,u) \quad = \quad 6 p^4 + 3 p^3 u + p^2 u^2 . $$
For special choices of the coefficients $c_{ij}$ in  (\ref{eq:ex24A}),
    some triples of lines in the arrangement $X \cap \mathcal{H}$ may meet in a point.
    For such matroids,  the ML degree drops from $6$ to some integer between $0$ and $5$.
  We recommend it as an exercise to the reader to explore these cases.
 For instance, can you find explicit $c_{ij}$ so that the ML degree of
 $X$ equals $3$? What are the prime ideal and the bidegree of $\mathcal{L}_X$ in that case? 
 How can the ML degree of $X$ be $0$ or~$1$?
  \hfill $\diamondsuit$
\end{example}

It would be interesting to know which statistical model $X$ in $\mathbb{P}^n$ defines the likelihood correspondence $\sL_X$ which is a complete intersection in $\mathbb{P}^n \times \mathbb{P}^n$.
When $X$ is a linear subspace of $\mathbb{P}^n$, this question is closely related
to the concept of {\em freeness} of a hyperplane arrangement.

 \begin{proposition}\label{CI}
If the hyperplane arrangement $X \cap \sH$ in $X$ is free, then the likelihood correspondence $\mathcal{L}_X$ is an ideal-theoretic complete intersection in $\mathbb{P}^n \times \mathbb{P}^n$.
\end{proposition}

\begin{proof}
For the definition of freeness see \S 1 in the paper
\cite{CDFV} by Cohen, Denman, Falk and Varchenko.
The proposition  is implied by their  \cite[Theorem 2.13]{CDFV} and \cite[Corollary 3.8]{CDFV}.
\end{proof}

Using
Theorem \ref{LinearCase}, this provides a likelihood geometry proof of Terao's theorem that the characteristic polynomial of a free arrangement factors into integral linear forms \cite{Terao}.

\section{Second Lecture}

In our newspaper we frequently read about studies aimed at
proving that a  behavior or food causes a certain 
medical condition. We begin the second lecture with an
introduction to statistical issues arising in such studies.  
The ``medical question''  we wish to address is {\em Does Watching Soccer on TV Cause Hair Loss?}
\ We learned this amusing example from  \cite[\S 1]{MSS}.

In a fictional study,
$296$ British subjects aged between $40$ to $50$ were interviewed about their hair length
and how many hours per week they watch soccer (a.k.a.~``football'') on TV.
Their responses are summarized in the  following {\em contingency table}
of format  $3 \times 3$:
$$ U \quad \, = \,\,\,\,
\bordermatrix{
 &  \hbox{lots of hair} & \hbox{medium hair} & \hbox{little hair} \cr
\hbox{$\leq 2$ hrs} & 51 & 45 & 33 \cr
\hbox{$2$--$6$ hrs} &  28 & 30 & 29 \cr
\hbox{$\geq 6$ hrs} & 15 & 27 & 38 }
    $$
    For instance, $29$ respondents reported having little hair and watching between
    $2$ and $6$ hours of soccer on TV per week.
    Based on these data, are these two random variables independent, or
    are we inclined to believe that watching soccer on TV and hair loss are correlated?

On first glance, the latter seems to be the case. Indeed, being independent
means that the data matrix $U$ should be close to a rank $1$ matrix. However,
all $2 \times 2$-minors of $U$ are strictly positive, indeed by quite a margin, and this
suggests a positive correlation.

However, this interpretation is deceptive. A much better explanation of our data can be given by
identifying a certain hidden random variable. That hidden variable
is {\em gender}. Indeed, suppose that among the respondents
$126$ were males and $170$ were females.
Our data matrix $U$ is then the sum of the male table and the female table, maybe as follows:
\begin{equation}
\label{eq:genderdecomposition}
U \,\, = \,\,
\begin{pmatrix}
 3 & 9 & 15 \\
  4 & 12 & 20 \\
  7 & 21 & 35 
  \end{pmatrix} \,\, +\,\,
\begin{pmatrix}
 48 & 36 & 18 \\
 24 & 18 & 9  \\
    8 &  6 &  3  
    \end{pmatrix}.
    \end{equation}
Both of these tables have rank $1$, hence $U$ has rank $2$.
Hence, the appropriate null hypothesis $H_0$ for analyzing our situation is not
independence but it is  {\em conditional independence}:
$$ H_0: \,\,\, \hbox{\em  Soccer on TV
 and  Hair Loss are  Independent  given Gender.} $$
And, based on the data $U$, we most definitely do not reject that null hypothesis.

\smallskip

The key feature of the matrix $U$ above was that it has rank $2$.
We now define low rank matrix models in general.
Consider two discrete random variables $X$ and $Y$ having $m$ and $n$
states respectively. Their {\em joint
probability distribution} is written as an $m \times n$-matrix
$$
P \quad = \quad \begin{pmatrix}
p_{11} & p_{12} & \cdots & p_{1n} \\
p_{21} & p_{22} & \cdots & p_{2n} \\
 \vdots & \vdots & \ddots & \vdots \\
p_{m1} & p_{m2} & \cdots & p_{mn} \\
\end{pmatrix}
$$
whose entries are nonnegative and sum to $1$.
Here $p_{ij}$ represents the probability that $X$ is in state $i$
and $Y$ is in state $j$. The of all probability distributions
is the standard simplex $\Delta_{mn-1}$ of dimension $mn-1$.
We write $\mathcal{M}_r$ for the  manifold of  rank $r$ matrices in $\Delta_{mn-1}$.

The matrices $P$ in  $\mathcal{M}_1$ represent  independent distributions.
Mixtures of $r$ independent distributions correspond to matrices in $\mathcal{M}_r$.
As always in applied algebraic geometry, we can make any problem 
that involves semi-algebraic sets progressively easier by three steps:
\begin{itemize}
\item disregard inequalities,
\item replace real numbers with complex numbers,
\item replace affine space by projective space.
\end{itemize}
In our situation, this leads us to replacing
$\mathcal{M}_r$ with its  Zariski closure 
 in complex  projective space $\PP^{mn-1}$.
 This Zariski closure is the projective variety  $\mathcal{V}_r$  of
complex $m {\times} n$ matrices of  rank $\leq r$. 
Note that $\mathcal{V}_r$ is singular along $\mathcal{V}_{r-1}$.
The codimension of  $\mathcal{V}_r$ is $(m-r)(n-r)$. It is a non-trivial exercise to  write 
the degree of $\mathcal{V}_r$ in terms of $m,n,r$. Hint: \cite[Example 15.2]{MS}.

Suppose now that i.i.d.~samples are drawn from an unknown joint
distribution on our two random variables $X$ and $Y$.
We summarize the resulting data in a contingency table
$$
U \quad = \quad \begin{pmatrix}
u_{11} & u_{12} & \cdots & u_{1n} \\
u_{21} & u_{22} & \cdots & u_{2n} \\
 \vdots & \vdots & \ddots & \vdots \\
u_{m1} & u_{m2} & \cdots & u_{mn} \\
\end{pmatrix}.
$$
The entries of the matrix $U$ are nonnegative integers whose sum is $u_{++}$.

The {\em likelihood function} for the contingency table $U$ is the following
function on $\Delta_{mn-1}$:
$$
P \quad \mapsto \quad
\binom{u_{++}}{u_{11} u_{12} \cdots u_{mn}}
\prod_{i=1}^m \prod_{j=1}^n p_{ij}^{u_{ij}} .
$$
Assuming fixed sample size, this is the likelihood of observing the data $U$
given an unknown  probability distribution $P$ in $\Delta_{mn-1}$.
In what follows we suppress the multinomial coefficient.
Furthermore, we regard the likelihood function as a rational function on $\PP^{mn-1}$,
so we write
$$ \ell_U \,\,= \,\, \frac{ \prod_{i=1}^m \prod_{j=1}^n p_{ij}^{u_{ij}}}{p_{++}^{u_{++}}}. $$
We wish to find a low rank probability matrix $P$ that best explains
the data $U$. Maximum likelihood estimation means solving
the following optimization problem:
\begin{equation}
\label{eq:matrixMLE}
\hbox{ Maximize 
$\,\ell_U(P)\,$ subject to $\,P \in \mathcal{M}_r $.}
\end{equation}
The optimal solution $\hat P$ is a rank $r$ matrix. 
This is the {\em maximum likelihood estimate} for $U$.

For $r = 1$, the independence model,
 the maximum likelihood estimate $\hat P$ is obtained
from the data matrix $U$ by the following formula,
already seen for $m=n=2$ in (\ref{eq:MLE22}).
Multiply the vector of row sums with the
vector of column sums and divide by the sample size:
\begin{equation}
\label{eq:rank1MLE}
\hat{P} \,\,\, = \,\,\,
\frac{1}{(u_{++})^2} \cdot
\begin{pmatrix}
u_{1+} \\ 
u_{2+} \\
\vdots \\
u_{m+} \end{pmatrix} \cdot
\begin{pmatrix} u_{+1} & u_{+2} & \cdots & u_{+n} \end{pmatrix}.
\end{equation}
Statisticians, scientists and engineers refer to such a formula
as an ``analytic solution''. In our view, it would be more appropriate
to call this an ``algebraic solution''. After all, we are here using
algebra not analysis. Our algebraic solution for $r=1$ reveals the following points:
\begin{itemize}
\item The MLE $\hat{P}$ is a {\em rational function} of the data $U$.
\item The function $U \mapsto \hat{P}$ is an algebraic function of degree $1$.
\item The ML degree of the independence model $\mathcal{V}_1$ equals $1$.
\end{itemize}
We next discuss the smallest case when the ML degree is larger than $1$.

\begin{example}\label{ex:33determinant} \rm
Let $m=n=3$ and $r=2$.
Our MLE problem is to
maximize
\[  \ell_U \,=\,
 (p_{11}^{u_{11}} p_{12}^{u_{12}} p_{13}^{u_{13}}
  p_{21}^{u_{21}} p_{22}^{u_{22}} p_{23}^{u_{23}}
   p_{31}^{u_{31}} p_{32}^{u_{32}} p_{33}^{u_{33}})/
   p_{++}^{u_{++}}
\]
 subject to the constraints $P \geq 0$ and ${\rm rank}(P) = 2$,
 where $P = (p_{ij})$ is a $3 {\times} 3$-matrix of unknowns.
 The equations that characterize the critical points of this optimization problem~are
 $$ {\rm det}(P)\,\, = \,\,\,
 \begin{small}
 \begin{matrix}
 \phantom{-} p_{11} p_{22} p_{33}
-  p_{11} p_{23} p_{32}
- p_{12} p_{21} p_{33} \\
+ p_{12} p_{23} p_{31}
+ p_{13} p_{21} p_{32}
- p_{13} p_{22} p_{31}
\end{matrix} \end{small}
\,\,\,=\,\,\, 0
  $$
and the vanishing of the $3 \times 3$-minors of the following $3 \times 9$-matrix:
$$ 
\begin{bmatrix}
 u_{11} & u_{12} & u_{13} & u_{21} & u_{22} & u_{23} & u_{31} & u_{32} & u_{33} \\
  p_{11} & p_{12} & p_{13} & p_{21} & p_{22} & p_{23} & p_{31} & p_{32} & p_{33} \\
   p_{11}a_{11}\! &\!\! p_{12}a_{12} \!\! & \!p_{13}a_{13} \!\! & 
      p_{21}a_{21}\!\! &\! p_{22}a_{22} \!\! & \!p_{33}a_{33} \!\! & 
      p_{31}a_{31}\!\! &\! p_{32}a_{32} \!\! & \!p_{33}a_{33}
\end{bmatrix}
$$
where $a_{ij} = \frac{\partial {\rm det}(P)}{\partial p_{ij}}$ is the cofactor of $p_{ij}$ in $P$.
For random positive data $u_{ij}$, these equations have $10$ solutions with ${\rm rank}(P) = 2$
in $\PP^8 \backslash \mathcal{H}$. Hence the ML degree of $\mathcal{V}_2$ is
$10$. If we regard the $u_{ij}$ as unknowns, then saturating the
above determinantal equations with respect to
$\,\mathcal{H} \cup \mathcal{V}_1\,$ yields the
prime ideal of  the likelihood correspondence
$\mathcal{L}_{{\mathcal{V}_2}} \subset \PP^8 \times \PP^8 $. 
See Example \ref{ex:manyclasses} 
for  the bidegree and other enumerative invariants of the
$8$-dimensional variety $\mathcal{L}_{\mathcal{V}_2}$.
\hfill $ \diamondsuit $
\end{example}

Recall from Definition \ref{def:MLD_and_LC} that
the {\em ML degree} of  a statistical model (or a projective variety)
is the number of critical points of the likelihood function for generic data.

\begin{theorem}
\label{thm:MLvalues}
The known values for the ML degrees of the determinantal varieties $\mathcal{V}_r$ are
$$
\begin{matrix}
        & \!(m,n) = \!\! & \!(3,3) & \!\!(3,4)\! & \!(3,5)\! & \! (4,4) & \! (4,5)  & (4,6)  & (5,5) \\
r=1 & &      1 &   1   &   1 & 1 &  1 &  1   & 1  \\
r=2 & &      10 & 26 & {\bf 58} & {\bf 191} & {\bf 843} &  {\bf 3119} &  \, {\bf 6776} \\
r=3 & &      1   & 1 & 1 & {\bf 191} &  {\bf 843} &  {\bf 3119} & {\bf 61326}\, \\
r=4 & &       & & &  1      &1 & 1 &\, {\bf 6776} \\
r=5 & &       & & &        &   &     &  \,\, 1 \\
\end{matrix}
$$
\end{theorem}

The numbers $10$ and $26$ were computed back in 2004 using 
the symbolic software {\tt Singular}, and
they were reported in \cite[\S 5]{HKS}. The bold face numbers
were found in 2012 in \cite{HRS} using the numerical software {\tt Bertini}.
In what follows we shall describe some of the details.

\begin{remark} \rm
Each determinantal variety $\mathcal{V}_r$ 
is singular along the smaller variety $\mathcal{V}_{r-1}$.
Hence, the very affine variety
$ \mathcal{V}_r \backslash \mathcal{H}$
is singular for $r \geq 2$, so Theorem \ref{thm:eulerchar} does not apply.
Here, $\mathcal{H} =  \{p_{++} \prod \! p_{ij} = 0 \} $.
According to Conjecture \ref{conj:eulerpositive},
the ML degree above provides a lower bound for the signed topological Euler characteristic of  $\mathcal{V}_r \backslash \mathcal{H}$.
The difference between the two numbers reflect the nature of the singular
locus $\,\mathcal{V}_{r-1} \backslash \mathcal{H}\,$
inside $\,\mathcal{V}_r \backslash \mathcal{H}$.
For plane curves that have nodes and cusps, we encountered this issue 
in Examples \ref{ex:nodecusp} and \ref{ex:cautionary}.
\end{remark}

We begin with a geometric description of the likelihood correspondence.
An $m \times n$-matrix $P$ is a regular point in $\mathcal{V}_r$ if and only if
${\rm rank}(P)= r$. 
The tangent
space $T_P$ is a subspace of dimension
$rn{+}rm{-}r^2$ in $\C^{m \times n}$.
Its orthogonal complement
 $T_P^\perp$ has dimension  $(m{-}r)(n{-}r)$.

The partial derivatives of the log-likelihood 
function  $ {\rm log}(\ell_U)\,$ on $\,\PP^{mn-1}\,$ are
$$
\frac{\partial { {\rm log}(\ell_U)}}{ \partial p_{ij} } \,\, = \,\, \frac{u_{ij} }{p_{ij}} - \frac{u_{++}}{p_{++}}  .
$$

\begin{proposition} \label{prop:critical}
An $m \times n$-matrix $P$ of rank $r$ is a critical point for
 ${\rm log}(\ell_U)$ on $\mathcal{V}_r$ if and only if
the linear  subspace  $T_P^\perp$ contains the matrix
$$ \left[ \,\frac{u_{ij} }{p_{ij}} - \frac{u_{++}}{p_{++}} \, \right]_{i=1,\ldots,m \atop j = 1,\ldots,n}$$
\end{proposition}

In order to get to the numbers in Theorem \ref{thm:MLvalues},
the geometric formulation was replaced 
in \cite{HRS} with a
parametric representation of the rank constraints.
The following linear algebra formulation worked well for
non-trivial computations.
 Assume $m \leq n$.
Let $P_1, R_1, L_1$ and $\Lambda$ be matrices
of unknowns of formats $r \times r$,
$r \times (n{-}r)$,  $(m{-}r) \times r$,  and $(n{-}r) \times (m {-} r)$. Set
$$
L = \begin{pmatrix} L_1 \!&\! -I_{m-r}\end{pmatrix}\! , \,\,\,
P = \begin{pmatrix} P_1  &\! P_1 R_1 \\ L_1 P_1 &\! L_1 P_1 R_1 \end{pmatrix} \!, \,\,\,
 \hbox{and} \,\,\,
R = \begin{pmatrix} R_1 \\ - I_{n-r} \end{pmatrix} \! ,
$$
where $I_{m-r}$ and $I_{n-r}$ are identity matrices. In the next statement
we use the symbol $\,\star\,$ for the Hadamard (entrywise) product of two matrices
that have the same format.

\begin{proposition}
Fix a general $m \times n$ data matrix $U$. The polynomial system 
$$  P \star (R \cdot \Lambda \cdot L)^T \,+\, u_{++} \cdot P \,\,=\,\, U
$$
consists of $mn$ equations in $mn$ unknowns. For generic $U$, it
 has finitely many complex solutions $(P_1,L_1,R_1,\Lambda)$. The 
$m {\times} n$-matrices $P$ resulting from these solutions are precisely the critical points  of the 
likelihood function $\ell_U$ on the determinantal variety $\mathcal{V}_r$.
\end{proposition}

We next present the analogue to
Theorem \ref{thm:MLvalues} for symmetric matrices
$$
P \quad = \quad \begin{pmatrix}
2p_{11} & p_{12} & p_{13} & \cdots & p_{1n} \\
p_{12} & 2 p_{22} & p_{23} & \cdots & p_{2n} \\
p_{13} & p_{23} & 2 p_{33} & \cdots & p_{3n} \\
 \vdots & \vdots & \vdots & \ddots & \vdots \\
p_{1n} & p_{2n} & p_{3n} & \cdots & 2 p_{nn} \\
\end{pmatrix}.
$$
Such matrices, with nonnegative coordinates $p_{ij}$ that sum to $1$,
represent joint probability distributions for two identically distributed
random variables with $n$ states.
The case $n= 2$ and $r=1$ is the Hardy-Weinberg curve, which we
discussed in detail in Example \ref{ex:hardyweinberg}.

\begin{theorem}
\label{thm:MLSymmValues}
The known values for  ML degrees of symmetric matrices of rank at most $r$
(mixtures of $r$ independent identically distributed random variables) are
$$
\begin{matrix}
        & n\,\, = &\, 2 \,&\, 3 \,& 4 & 5 & 6 \\
r=1 & &         1 & 1 & 1 & 1 & 1  \\
r=2 & &         1 & 6 \,& {\bf 37} & {\bf 270} & {\bf 2341}\\
r=3 & &        & 1 & {\bf 37} & {\bf 1394}& ?\\
r=4 & &         &   &   1      & {\bf 270} & ?\\
r=5 & &         &   &          &   1       & {\bf 2341}\\
\end{matrix}
$$
\end{theorem}

At present we do not know the common value of the ML degree for
$n=6$ and $r=3,4$.
In what follows we take a closer look at the model
for symmetric $3 \times 3$-matrices of rank $2$.

\begin{example} \label{ex:galois33} \rm
Let $n=3 $ and $r=2$, so $X$ is a cubic hypersurface in $\PP^5$. The
likelihood correspondence $\mathcal{L}_X $
is a five-dimensional subvariety of $\PP^5 \times \PP^5$ having
bidegree 
$$ B_X(p,u) \,\, = \,\, 6 p^5+12 p^4 u+15 p^3 u^2+12 p^2 u^3+3 p u^4. $$
The bihomogeneous prime ideal of $\mathcal{L}_X$ is minimally generated by $23$ polynomials, namely:
\begin{itemize}
\item One polynomial of bidegree $(3,0)$; this is the determinant of $P$.
\item Three polynomials of degree $(1,1)$. These come from the underlying
   toric model \\ $\{{\rm rank}(P)=1 \}$. As suggested in Proposition   \ref{prop:toricLC},
   they are the
         $2 \times  2$-minors of
$$     \begin{pmatrix}
     2 p_0+p_1+p_2  & p_1+2 p_3+p_4  & p_2+p_4+2 p_5 \\
     2 u_0 + u_1 + u_2  & u_1 + 2 u_3+u_4  & u_2+u_4+2 u_5
\end{pmatrix}.    $$
\item One polynomial of degree $(2,1)$,
\item three polynomial of degree $(2,2)$,
\item nine polynomials of degree $(3,1)$,
\item six  polynomials of degree $(3,2)$.
\end{itemize}
It turns out that this ideal represents an expression 
for the MLE $\hat P$ in terms of radicals in~$U$.

We shall work this out for one
numerical example. Consider
the data matrix $U$ with
$$
u_{11} = 10, \,
u_{12} = 9, \,
u_{13}  = 1, \,
u_{22}  = 21,\,
u_{23} = 3,\,
u_{33}  = 7. $$       
For this choice, all six critical points of the likelihood function                                   
are real and positive:
$$
\begin{matrix} \smallskip
p_{11} & p_{12} & p_{13} & p_{22} & p_{23} & p_{33} && \,\,\,{\rm log} \,\ell_U(p) \\
0.1037 &  0.3623 & 0.0186 & 0.3179 & 0.0607 & 0.1368 &&    -82.18102 \\
0.1084 &  0.2092 & 0.1623 & 0.3997 & 0.0503 & 0.0702 &&   -84.94446 \\
\medskip
0.0945 &  0.2554 & 0.1438 & 0.3781 & 0.4712 & 0.0810  &&   -84.99184 \\
0.1794 &  0.2152 & 0.0142 &  0.3052 & 0.2333 &  0.0528 && -85.14678 \\
0.1565 &  0.2627 & 0.0125 & 0.2887 & 0.2186 & 0.0609 &&    -85.19415 \\
0.1636 &  0.1517 & 0.1093 & 0.3629 & 0.1811 & 0.0312 &&   -87.95759 \\
\end{matrix}
$$
The first three points are local maxima in $\Delta_5$ and the last
three points are local minima. These six points define an algebraic field extension
of degree $6$ over $\mathbb{Q}$. One might expect that the
Galois group of these six points over $\mathbb{Q}$ is the 
full symmetric group $S_6$. If this were the case then the above coordinates
could not be written in radicals.
However, that expectation is wrong.
The Galois group of the likelihood fibration ${\rm pr}_2:\mathcal{L}_X \rightarrow \PP^{5}_U$
given by the $3 \times 3$ symmetric problem is
a subgroup of $S_6$ isomorphic to the solvable group $S_4$.

To be concrete, for the data above,  the minimal polynomial
for the MLE  $\hat p_{33}$ equals
$$ \begin{matrix}
9528773052286944p_{33}^{6} - 4125267629399052p_{33}^{5} 
+ 713452955656677p_{33}^{4} \qquad \qquad \\ - 63349419858182p_{33}^{3}  
+ 3049564842009p_{33}^{2} - 75369770028p_{33} + 744139872 \, = \, 0.
\end{matrix}
$$
We solve this equation in radicals as follows:
 $$
 \begin{matrix}
 \medskip
  p_{33} & =&
\frac{16427}{227664} + \frac{1}{12}\!\left(\zeta - \zeta^{2}\right)\omega_2 
 - \frac{66004846384302}{19221271018849}\omega_2^2 +  \\ & & 
 \left(
  \frac{14779904193}{211433981207339} \zeta^{2}
- \frac{14779904193}{211433981207339}\zeta   \right)\omega_1\omega_2^2 + \frac{1}{2}\omega_3,
\end{matrix}
$$
 where $\zeta$ is a primitive third root of unity,
$\,\omega_1^2  =  94834811/3$, and
$$
\begin{matrix} \medskip
\omega_2^3 & \!\! =&  \left(\frac{5992589425361}{150972770845322208}\zeta - \frac{5992589425361}{150972770845322208}\zeta^{2}\right) + \frac{97163}{40083040181952}\omega_1,
\qquad \\ \smallskip
\omega_3^2 & \!\! = &\!\! \!  \frac{5006721709}{1248260766912} + \left(\frac{212309132509}{4242035935404}\zeta - \frac{212309132509}{4242035935404}\zeta^{2}\right)\! \omega_2 - \frac{2409}{20272573168}\omega_1\omega_2
\\ & &  - \frac{158808750548335}{76885084075396}\omega_2^2
+ \left(  \frac{17063004159}{422867962414678}\zeta^{2} - \frac{17063004159}{422867962414678}\zeta \right)\omega_1\omega_2^2 .
\end{matrix}
$$ 
The explanation for the extra symmetry 
stems from the duality theorem below. It furnishes an involution on the
set of six critical points that allows us to express them in radicals.
\hfill $\diamondsuit$
\end{example}

The tables in Theorems \ref{thm:MLvalues} and \ref{thm:MLSymmValues}
suggest that the columns will always be symmetric.
This fact was conjectured in \cite{HRS} and subsequently proved 
by Draisma and Rodriguez in \cite{DR}.

\begin{theorem}
\label{conj:duality}
Fix $m \leq n$ and consider the determinantal varieties $\mathcal{V}_i$ for either general or symmetric matrices.
Then the ML degrees for rank $r$ and for rank $m{-}r{+}1$~coincide.
\end{theorem}

In fact, the main result in \cite{DR} establishes the following more precise statement.
Given a data matrix $U$ of format $m \times n$,  we write 
$\,\Omega_U$ for the 
$m \times n$-matrix
whose $(i,j)$ entry equals
$$ \frac{u_{ij} \cdot u_{i+} \cdot u_{+j}}{(u_{++})^3}. $$

\begin{theorem} \label{thm:duality2}
Fix $m \leq n$ and $U$ an $m \times n$-matrix with strictly positive integer
entries. There exists a bijection
 between the complex critical points $P_1,P_2,\ldots,P_s$
 of the likelihood function $\,\ell_U$ on $\mathcal V_r$ and the complex
critical points $Q_1,Q_2,\ldots, Q_s$ of $\,\ell_U$ on $\mathcal V_{m-r+1}$
 such that
 $$
 P_1 \star Q_1 \, = \,  P_2 \star Q_2  \, = \, \, \cdots \,\,=\,
 P_s \star Q_s \,\, = \,\,  \Omega_U.
$$
 Thus, this bijection preserves reality, positivity, and
rationality.
 \end{theorem}

The key to computing the ML degree tables and to formulating
the duality conjectures in \cite{HRS}, was the use of
numerical algebraic geometry. The software {\tt Bertini}
allowed for the computation of  thousands of instances
in which the formula of Theorem \ref{thm:duality2} was confirmed.

 {\tt Bertini}  is  numerical software, based on homotopy continuation, for finding all complex
solutions to a system of polynomial equations (and much more).
The software is available at \cite{Bertini}.
The developers,
Daniel  Bates, Jonathan Hauenstein, Andrew  Sommese, Charles Wampler, 
have just completed a new textbook \cite{BHSW} on the
mathematics behind {\tt Bertini}.

For the past two decades, algebraic geometers have increasingly
employed computational methods as a tool for their research. However, these
computations have almost always been symbolic (and hence exact). They
relied on Gr\"obner-based
software such as {\tt Singular} or {\tt Macaulay2}. Algebraists often feel
a certain discomfort when asked to trust a numerical computation.
We encourage discussion about this issue, by raising the following question.

\begin{example} \rm
In the rightmost column of  Theorem \ref{thm:MLSymmValues}, it is
asserted that the solution to a certain enumerative geometry problem is {\bf 2341}.
Which of these would {\bf you} trust most:
\begin{itemize}
\item the output of a symbolic computation?
\item the output of a numerical computation?
\item a proof written by an algebraic geometer?
\end{itemize}
In the authors' view, it always pays off to be critical
and double-check all computations, regardless
of how they were carried out. And, this applies to all three of the above. 
\hfill $\diamondsuit$
\end{example}

One of the big advantages of numerical algebraic geometry
over Gr\"obner bases when it comes to MLE is
the separation between {\em Preprocessing} and {\em Solving}.
For any particular variety $X \subset \PP^n$, such
as $X = \mathcal{V}_r$, we preprocess by solving the
likelihood equations once, for a generic data set $U_0$
chosen by us. The coordinates of $U_0$ may be
complex (rather than real) numbers. We can chose them with stable
numerics in mind, so as to compute all critical points
up to high accuracy. This step can take a long time,
but the output is highly reliable.

After solving the equations once, for that
generic $U_0$, all subsequent computations
for any other data set $U$ are very fast. In particular, the
computation is fully parallelizable. If we have $m$ processors
at our disposal, where $m = {\rm MLdegree}\,(X)$,
then each processor can track one of the 
paths. To be precise,  homotopy continuation starts from the critical points 
of $\ell_{U_0}$ and transform them into the
critical points of $\ell_{U}$.
Geometrically speaking, for fixed $X$,
 the homotopy amounts to 
 walking on the sheets of the likelihood fibration
 ${\rm pr}_2:\mathcal{L}_X \rightarrow \mathbb{P}^n_u$.
 
To illustrate this point, here are the 
timings (in seconds) that were reported in \cite{HRS}
for the determinantal varieties $X = \mathcal{V}_r$.
Those computations were carried out in {\tt Bertini} on a 64-bit Linux cluster with $160$ processors.
The first row is the preprocessing time for solving the equations once.
The second row is the time needed to solve
any subsequent instance:
\begin{table}[h]
  \centering
  \begin{tabular}{ccccccc}
  $(m,n,r)$ & $\!\!(4,4,2)\!\!$ & $\!\!(4,4,3)\!\!$ & $\!\!(4,5,2)\!\!$ & $ \!\! (4,5,3) \!\! $ & $\!\!(5,5,2)\!\!$ & $\!\!(5,5,4)$ \\
  \hline
 \!\!\! Preprocessing \!\! & 257 & 427 & 1938 & 2902 & 348555 &   \!\!146952   \\
  Solving & 4 & 4 & 20 & 20 & 83 & 83  \\
\end{tabular}
\end{table}

\noindent
This table suggests that combining  numerical algebraic geometry
with existing tools from computational statistics might lead to a 
viable tool for certifiably solving MLE problems.

\smallskip

We are now at  the point where it is essential to offer a disclaimer.
The low rank model $\mathcal{M}_r$ does not correctly represent the
notion of conditional
independence. The model we should have used instead is the
 {\em mixture model} ${\rm Mix}_r$. By definition, ${\rm Mix}_r$ is
  the set of probability distributions
$P$ in $\Delta_{mn-1}$ 
that are convex combinations of $r$ independent distributions, 
each taken from $\mathcal{M}_1$.
Equivalently, the mixture model ${\rm Mix}_r$ consists of all~matrices
\begin{equation}
\label{eq:mixformula}
 P \,\,=\,\, A \cdot \Lambda \cdot B,
\end{equation}
where $A$ is a nonnegative $m {\times} r$-matrix whose
rows sum to $1$, 
$\Lambda$ is a nonnegative $r {\times} r$ diagonal matrix
whose entries sum to $1$, 
and $B$ is a nonnegative $r {\times} n$-matrix whose
columns sum to $1$. The formula (\ref{eq:mixformula}) expresses
${\rm Mix}_r$ as the image of a trilinear map between polytopes:
$$ \phi \,: \,(\Delta_{m-1})^r \times \Delta_{r-1} \times (\Delta_{n-1})^r \,\rightarrow \,
\Delta_{mn-1}\,, \quad
 (A, \Lambda, B)\,\,\, \mapsto\,\,\, P. $$
The following result is well-known; see e.g.~\cite[Example 4.1.2]{LiAS}.

\begin{proposition} \label{prop:mix}
Our low rank model $\mathcal{M}_r$ is the Zariski closure 
of the mixture model $\,{\rm Mix}_r\,$ in the probability simplex $\Delta_{mn-1}$. 
If $\,r \leq 2\,$ then $\,{\rm Mix}_r = \mathcal{M}_r$.
If $\,r \geq 3\,$ then $\,{\rm Mix}_r \subsetneq \mathcal{M}_r$.
\end{proposition}

The point here is the distinction between the
rank and the nonnegative rank of a nonnegative matrix.
Matrices in $\mathcal{M}_r$ have rank $\leq r$ and matrices
in ${\rm Mix}_r$ have nonnegative rank $\leq r$.
Thus elements of $\mathcal{M}_r \backslash {\rm Mix}_r $
are matrices whose nonnegative rank exceeds its rank.

\begin{example} \rm
The following $4 \times 4$-matrix has rank $3$ but nonnegative rank $4$:
$$ P \quad = \quad \frac{1}{8} \cdot \begin{pmatrix} 
1 & 1 & 0 & 0 \\
0 & 1 & 1 & 0 \\
0 & 0 & 1 & 1 \\
1 & 0 & 0 & 1 
\end{pmatrix}
$$
This is the slack matrix of a regular square.
It is an element of $\mathcal{M}_3 \backslash {\rm Mix}_3$.
\hfill $\diamondsuit$
\end{example}

Engineers and scientists care more about ${\rm Mix}_r$ than $\mathcal{M}_r$.
In many applications,  nonnegative rank is more relevant than  rank.
The reason can be seen in (\ref{eq:genderdecomposition}).
In such a low-rank decomposition, we do not want the female table or  the male table
to have a negative entry.

This raises the following important questions:
How to maximize the likelihood function $\ell_U$ over ${\rm Mix}_r$? 
What are the algebraic degrees associated with that optimization problem?

\smallskip

Statisticians seek to maximize the likelihood function $\ell_U$ on ${\rm Mix}_r$ 
by using the {\em expectation-maximization} (EM) algorithm in  the space 
$(\Delta_{m-1})^r \times \Delta_{r-1} \times (\Delta_{n-1})^r$
of parameters $(A, \Lambda, B)$.
In each iteration, the EM algorithm strictly decreases the 
{\em Kullback-Leibler divergence} from the current model point
$P = \phi(A,\Lambda,B)$ to the empirical distribution
$\,\frac{1}{u_{++}}\cdot U$. The hope in running the
EM algorithm for given data $U$ is that it
converges to the global  maximum $\hat P$ on ${\rm Mix}_r$.
For a presentation of
the EM algorithm for discrete algebraic models
see \cite[\S 1.3]{PS}. A  study of
the geometry of this algorithm for the mixture model ${\rm Mix}_r$
is undertaken in \cite{KRS}.

If the EM algorithm converges to a point that 
lies in the interior of the parameter polytope, and
is non-singular with respect to $\phi$, then that point will be 
among the critical points on $\mathcal{M}_r$.
These are  characterized by Proposition  \ref{prop:critical}.
However, since  ${\rm Mix}_r$ is properly contained in $\mathcal{M}_r$,
it frequently happens that the true MLE $\,\hat P\,$ lies on the boundary of ${\rm Mix}_r$.
In that case, $\hat P$ is not a critical point of $\ell_U$ on $\mathcal{M}_r$,
meaning that $(\hat P,U)$ is not in  the likelihood  correspondence
on $\mathcal{V}_r$. Such points will never be found by the method described above.

In order to address this issue, we need to identify the divisors in
the variety
 $\mathcal{V}_r \subset \PP^{mn-1}$ that appear in the
algebraic boundary of $ {\rm Mix}_r$.
By this we mean the irreducible components 
$W_1, W_2, \ldots, W_s$ of the Zariski closure of
$\partial {\rm Mix}_r$.
Each of these $W_i$ has codimension $1$ in $\mathcal{V}_r$.
Once the $W_i$ are identified, one would need to
examine their ML degree, and also the ML degree
of the various strata $W_{i_1} \cap \cdots \cap W_{i_s}$ 
in which $\ell_U$ might attain its maximum.
At present we do not have this information
even in the smallest non-trivial case $m=n=4$ and $r=3$.

\begin{example} 
\label{ex:44nonneg}
\rm
We illustrate this issue by describing one of the
components $W$ of the algebraic boundary for the
mixture model ${\rm Mix}_3$ when $m=n=4$.
Consider the equation
$$
\begin{pmatrix}
p_{11} & p_{12} & p_{13}  & p_{14} \\
p_{21} & p_{22} & p_{23}  & p_{24} \\
p_{31} & p_{32} & p_{33}  & p_{34} \\
p_{41} & p_{42} & p_{43}  & p_{44}
\end{pmatrix}
 \,= \,
\begin{pmatrix}
0  & a_{12}  & a_{13} \\
0 & a_{22} & a_{23} \\
a_{31} & 0 & a_{33} \\
a_{41} & a_{42}  & 0
\end{pmatrix} \cdot
\begin{pmatrix}
0 & b_{12} & b_{13} & b_{14} \\
b_{21} & 0 & b_{23}  & b_{24} \\
b_{31} & b_{32} & b_{33} & 0
\end{pmatrix}
$$
This parametrizes a
$13$-dimensional subvariety $W$ of
 the hypersurface
$\mathcal{V}_3 = \{{\rm det}(P) = 0\}$
in $\PP^{15}$.
The variety $W$ is a component in the algebraic boundary of
${\rm Mix}_3$. To see this, we choose the $a_{ij}$
and $b_{ij}$ to be positive, and we note that $P$ lies outside
${\rm Mix}_3$ when precisely one of the $0$ entries
gets replaced by  $-\epsilon$.
The prime ideal of $W$ in $\mathbb{Q}[p_{11}, \ldots,p_{44}]$ is obtained by
eliminating the $17$ unknowns $a_{ij}$ and $b_{ij}$ from the $16$ scalar equations.
A direct computation with {\tt Macaulay 2} shows that
the variety $W$ is Cohen-Macaulay of codimension-$2$.
By the Hilbert-Burch Theorem,  it is defined
   by the $4 \times 4$-minors of the $4 \times 5$-matrix.
   This following specific matrix representation 
was suggested to us by  Aldo Conca and Matteo Varbaro:
$$ \begin{pmatrix}
p_{11} &  p_{12} &  p_{13} &  p_{14} &  0 \\
p_{21} &  p_{22} &  p_{23} &  p_{24} &  0 \\
p_{31} &  p_{32} &  p_{33} &  p_{34} &   p_{34} (p_{11}p_{22}-p_{12}p_{21})\\
p_{41} &  p_{42} &  p_{43} &  p_{44} &   p_{41} (p_{12} p_{24} - p_{14} p_{22}
) +p_{44} (p_{11}p_{22}-p_{12}p_{21})  \\
\end{pmatrix}. $$ 
Tte algebraic boundary of ${\rm Mix}_3$ consists of precisely
$304$ irreducible components,  namely the $16$ coordinate hyperplanes
and $288$ hypersurfaces that are all isomorphic to $W$. This is proved in the
forthcoming paper \cite{KRS}.
At present, we do not know the ML degree of $W$.
\hfill $\diamondsuit$
\end{example}

The definition of rank varieties and mixture models extends to
$m$-dimensional tensors $P$ of arbitrary format
$d_1 \times d_2 \times \cdots \times d_m$.
We refer to Landsberg's book \cite{Land} for an introduction to tensors
and their rank. Now,  $\mathcal{V}_r$ is
the variety of tensors of borderrank $\leq r$, the model
$\mathcal{M}_r$ is the set of all probability distributions in $\mathcal{V}_r$, 
and the model ${\rm Mix}_r$ is the subset of tensors of nonnegative rank $\leq r$.
Unlike in the matrix case $m = 2$, 
the mixture model for borderrank $r = 2$  is already quite interesting 
when  $m \geq 3$.
We state two theorems that characterize our objects.
The set-theoretic version of Theorem \ref{thm:manivel} is due to
Landsberg and Manivel \cite{LM}. The
 ideal-theoretic statement was proved more recently by Raicu \cite{Rai}.
 
\begin{theorem} \label{thm:manivel}
The variety $\mathcal{V}_2$ is defined by the $3 \times 3$-minors
of all flattenings of $P$.
\end{theorem}
 
 Here, {\em flattening} means picking any subset $A $ of
$[n]=\{1,2,\ldots,n\}$ with $1 \leq |A| \leq n-1$ and writing the
tensor $P$ as an ordinary matrix with $\prod_{i \in A}d_i$ rows and
$\prod_{j \not\in A } d_j$ columns. 

\begin{theorem}
The mixture model $\,{\rm Mix}_2$ is the subset of supermodular distributions in $\mathcal{M}_2$.
\end{theorem}

This theorem was proved in \cite{ARSZ}.
Being {\em supermodular} means that $P$
satisfies a  natural family of quadratic binomial inequalities.
We explain these for $m=3, d_1 = d_2 = d_3 = 2$.

\begin{figure}[h] \label{fig:slice}
\includegraphics[scale=.28]{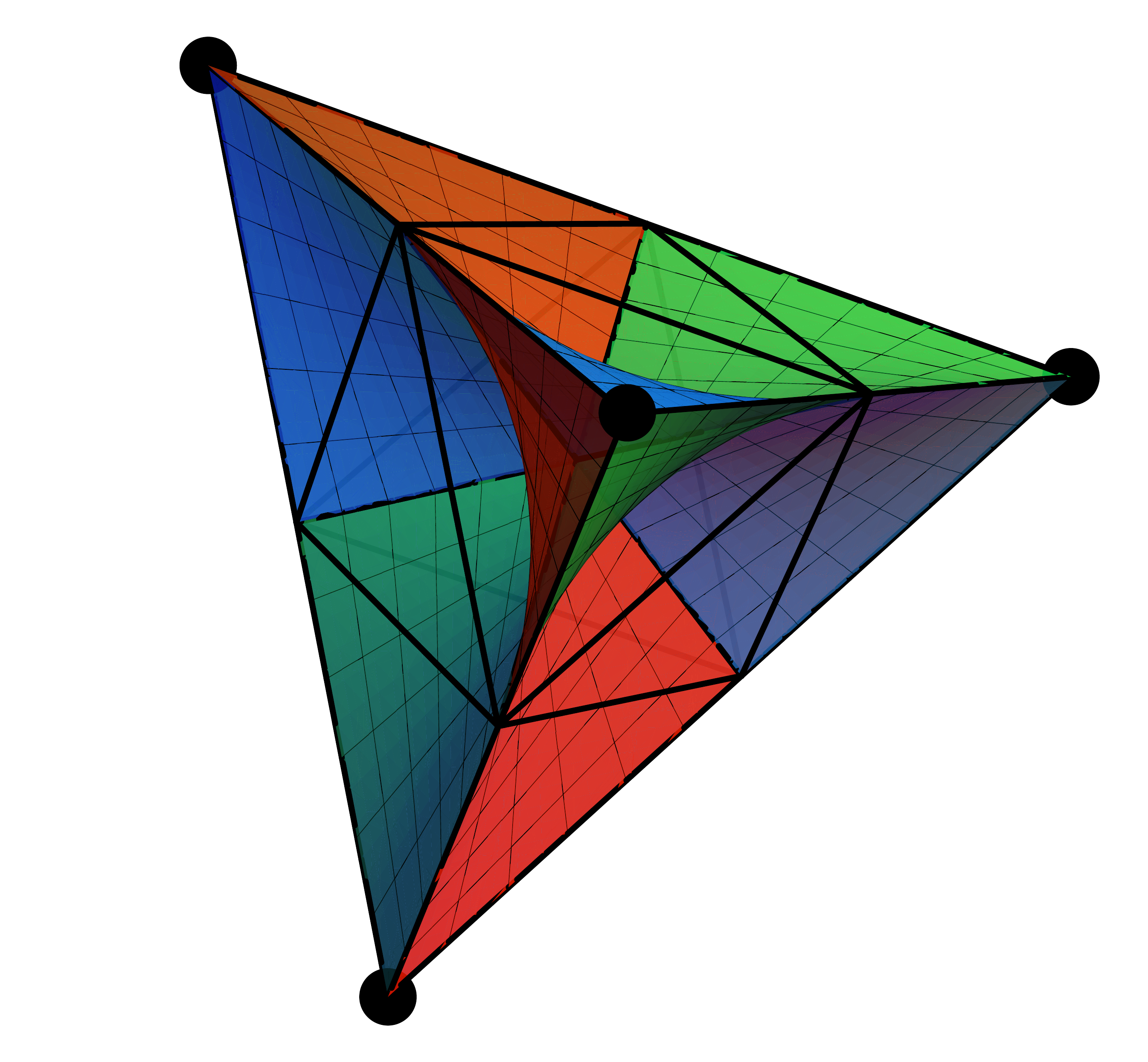} 
\includegraphics[scale=.31, ,trim=210 200 180 0, clip=true]{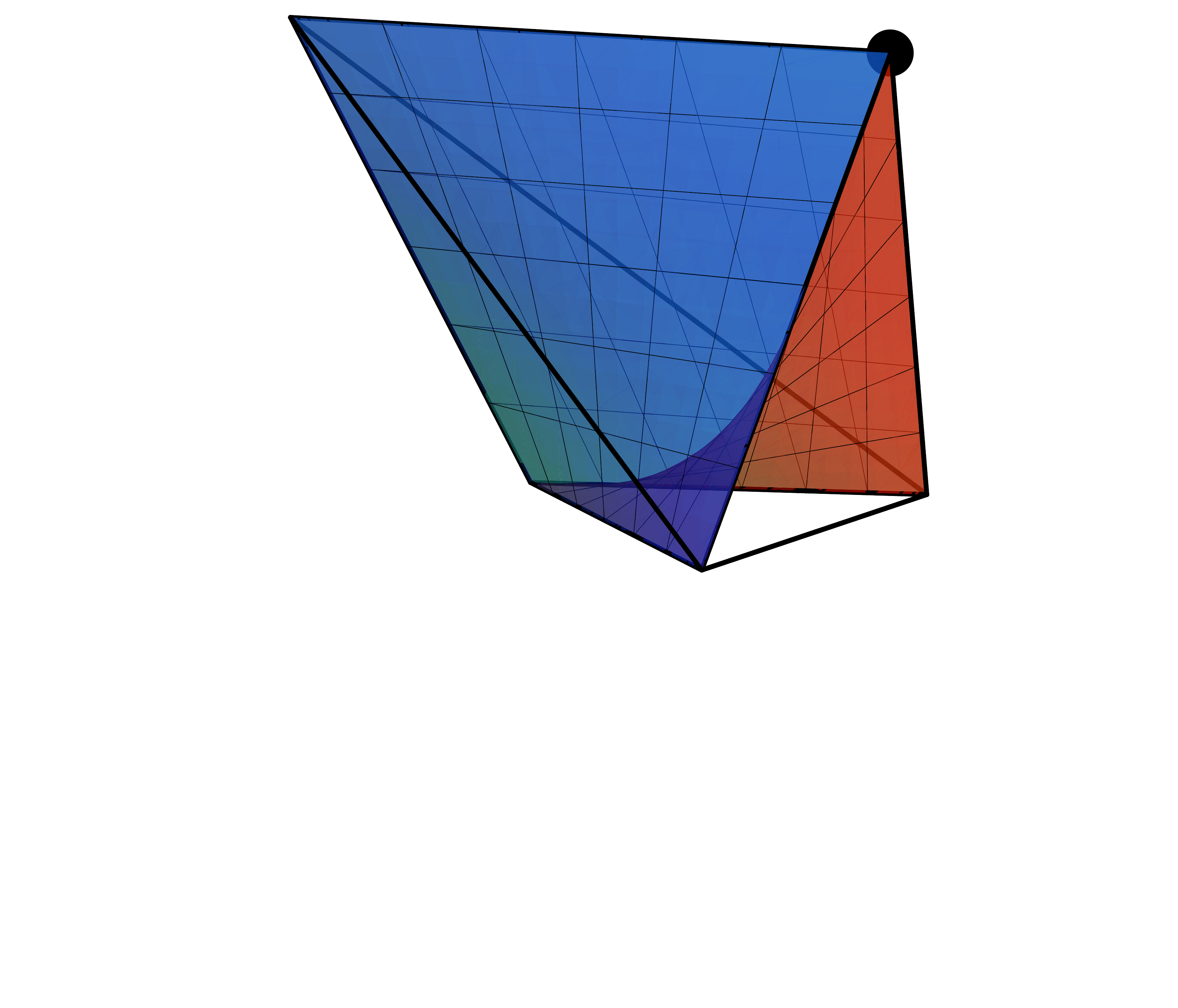}
\caption{A $3$-dimensional slice of the 
$7$-dimensional model of $2 {\times} 2 {\times} 2$ tensors of nonnegative rank $\leq                                                                                 
2$. Each toric cell is bounded by $3$ quadrics and contains a vertex
of the tetrahedron.}
\end{figure}

\begin{example} \rm
We consider $2 \times 2 \times 2$ tensors.
Since secant lines of the Segre variety $\PP^1 {\times} \PP^1 {\times} \PP^1 $ fill all of
$\PP^7$, we have that $\mathcal{V}_2 = \PP^7$ and $\mathcal{M}_2 = \Delta_7$. 
The mixture model ${\rm Mix}_2$ is an interesting, full-dimensional, closed, semi-algebraic 
subset of $\Delta_7$. By definition, 
${\rm Mix}_2$ is the image of a 2-to-1 map
$\phi : (\Delta_1)^7 \rightarrow \Delta_7$ analogous to 
(\ref{eq:mixformula}). 
   The branch locus
is the  $2 {\times} 2 {\times} 2$-hyperdeterminant, which is
a hypersurface in $\PP^7$ of degree $4$ and ML degree $13$.

The analysis in \cite[\S 2]{ARSZ} 
represents the  model ${\rm Mix}_2$ as the union of   four {\em toric cells}. One
of these toric cells is the set of tensors satisfying
\begin{equation}
\label{eq:nineineq}
\begin{matrix}
p_{111} p_{222} \geq p_{112} p_{221} & \quad
p_{111} p_{222} \geq p_{121} p_{212} & \quad
p_{111} p_{222} \geq p_{211} p_{122} \\
p_{112} p_{222} \geq p_{122} p_{212} & \quad
p_{121} p_{222} \geq p_{122} p_{221} & \quad
p_{211} p_{222} \geq p_{212} p_{221} \\
p_{111} p_{122} \geq p_{112} p_{121} & \quad
p_{111} p_{212} \geq p_{112} p_{211} & \quad
p_{111} p_{221} \geq p_{121} p_{211}
\end{matrix}
\end{equation}
A nonnegative $2 {\times} 2 {\times} 2$-tensor $P$ in $\Delta_7$ is {\em supermodular} 
if it satisfies these inequalities, possibly after label swapping $1 \leftrightarrow 2$.
We visualize  ${\rm Mix}_2$ by
restricting to the $3$-dimensional subspace $H$ given by
$\,p_{111} = p_{222}, \,
p_{112} = p_{221}, \,
p_{121} = p_{212}\,$ and $\,
p_{211} = p_{122}$.
The intersection $H \cap \Delta_7$ is a  tetrahedron,
and we consider $H \cap {\rm Mix}_2$ inside that tetrahedron.
The restricted model $H \cap {\rm Mix}_2$ is shown on the
left in Figure 1. It consists of four
toric cells as shown on the right side.
The boundary is given by three quadratic
surfaces, shown in red, green and blue, and
which are obtained from either the first or the
second row in (\ref{eq:nineineq}) by restriction to $H$.

The boundary analysis suggested in
Example \ref{ex:44nonneg} turns out to be quite
simple in the present example.
All boundary strata of  the model ${\rm Mix}_2$
are varieties of ML degree $1$. 

One such boundary stratum for ${\rm Mix}_2$ is the
$5$-dimensional toric variety
$$ X \,=\, V( p_{112} p_{222}-p_{122} p_{212}, p_{111} p_{122} -p_{112} p_{121},
p_{111} p_{222}-p_{121} p_{212})
\,\,\, \subset \,\,\, \PP^7. $$
As a preview for what is to come, we report
its ML bidegree and its sectional ML degree:
\begin{equation}
\label{eq:BSpreview}
\begin{matrix}
B_X(p,u) & = & p^7\,+\, 2 p^6 u\,+\,3p^5 u^2\,+\,3p^4u^3\,+\,3 p^3u^4\,+\,3p^2u^5, \\
S_X(p,u) & = & p^7+ 14 p^6 u+30 p^5 u^2+30 p^4u^3+15p^3u^4+3p^2u^5.
\end{matrix}
\end{equation}
In the next section, we shall study the class of toric varieties
and the class of varieties having  ML degree 1. Our variety $X$
lies in the intersection of these two important classes.
\hfill $\diamondsuit$
\end{example}

\section{Third Lecture}

In our third lecture we start out with the
likelihood geometry of embedded toric varieties.
Fix a $(d{+}1) \times (n{+}1)$ integer matrix $A = (a_0,a_1,\ldots,a_n)$ of rank $d{+}1$ 
that has $(1,1,\ldots,1)$ as its last row. This matrix defines an effective action
of the torus $\,(\mathbb{C}^*)^d$ on projective space~$\PP^n$:
\[
(\mathbb{C}^*)^d \times \mathbb{P}^n \longrightarrow \mathbb{P}^n, \qquad t \times (p_0:p_1:\cdots:p_n) \longmapsto (t^{\tilde a_0} \cdot p_0:t^{\tilde a_1} \cdot p_1:\cdots:t^{\tilde a_n} \cdot p_n).
\]
Here $\tilde a_i$ is the column vector $a_i$ with the last entry $1$ removed.
We
also fix  
\[
c = (c_0,c_1,\ldots,c_n)  \in (\mathbb{C}^*)^{n+1},
\]
viewed as a point in $\PP^n$.
Let $X_c $ be the closure in $\PP^n$ of the orbit
$\,(\mathbb{C}^*)^d \cdot c$. 
This is a projective toric variety of dimension $d$, defined
by the pair $(A,c)$.
The ideal that defines $X_c$ is the
familiar {\em toric ideal} $I_A$ as in \cite[\S 1.3]{LiAS}, but
with $p = (p_0,\ldots,p_n)$ replaced by
\begin{equation}
\label{eq:poverc}
 p/c \,\, = \,\,
\biggl(
\frac{p_0}{c_0},
\frac{p_1}{c_1},\ldots,
\frac{p_n}{c_n}
\biggr).
\end{equation}

\begin{example} \label{ex:cubicsurface}
\rm
Fix $d = 2$ and $n=3$. The matrix
$$ A = \begin{pmatrix}
0 &  3 & 0 & 1 \\ 
0 &  0 & 3 & 1 \\
1 &  1 & 1 & 1
\end{pmatrix}
$$
specifies the following family of toric surfaces of degree three in $\PP^3$:
$$ 
X_c \,\,=  \,\,
\overline{\{ (c_0:c_1x_1^3:c_2x_2^3:c_3x_1x_2) \,: (x_1,x_2) \in (\mathbb{C}^*)^2 \} } \\
\,\,= \,\, V(c_3^3 \cdot p_0 p_1 p_2\, -\, c_0c_1c_2 \cdot p_3^3).
$$
Of course, the prime ideal of any particular surface $X_c$ is the principal  ideal generated by
$$ 
\frac{p_0}{c_0} \frac{p_1}{c_1} \frac{p_2}{c_2}  \,-\,
\biggl(\frac{p_3}{c_3} \biggr)^3. $$
How does the ML degree of $X_c$ depend on the parameter 
$c = (c_0,c_1,c_2,c_3) \in (\mathbb{C}^*)^4$?
\hfill $\diamondsuit$
\end{example}

We shall express the ML degree of the toric variety $X_{c}$
in terms of the complement of a hypersurface in the
torus $(\mathbb{C}^*)^d$.
The pair $(A,c)$ define the sparse Laurent polynomial
$$ f(x) \,\,=\,\, c_0 \cdot x^{\tilde a_0} + c_1 \cdot x^{\tilde a_1} + \cdots + c_n \cdot x^{\tilde a_n}.$$

\begin{theorem} \label{thm:toricMLE}
The ML degree of the $d$-dimensional toric variety $\,X_c \subset \PP^n\,$ is equal to  $(-1)^d$
times the Euler characteristic of the very affine variety
\begin{equation}
\label{eq:hscomp}
 X_c \backslash \mathcal{H} \,\,\, \simeq \,\,\,
\bigl\{ x \in (\mathbb{C}^*)^d \,: \, f(x) \not= 0 \bigr\}. 
\end{equation}
For generic $c$, the ML degree  agrees with the degree of $X_{c}$, which is the
normalized volume of the $d$-dimensional lattice polytope ${\rm conv}(A)$
obtained as the convex hull of the columns of~$A$.
\end{theorem}

\begin{proof} We first argue that the identification (\ref{eq:hscomp}) holds.
The map 
\[ x \,\,\longmapsto \,\, {p} \,=\, (c_0 \cdot{ x}^{\tilde a_0}:
 c_1 \cdot x^{\tilde a_1} : \cdots : c_n \cdot  { x}^{\tilde a_n}) \]
defines an injective group homomorphism from 
$(\mathbb{C}^*)^d$ into the dense torus of $\PP^n$. Its image
is equal to the dense torus of $X_{c}$, so we have an isomorphism
between $(\mathbb{C}^*)^d$ and the dense torus of $X_{ c}$.
Under this isomorphism, the affine open set $\{f \not= 0\}$ in $(\mathbb{C}^*)^d$
is identified with the affine open set $\{p_0+\cdots + p_n \not= 0\}$
in the dense torus of $X_{c}$. The latter is precisely
$X_{ c} \backslash \mathcal{H}$.
Since $(\mathbb{C}^*)^d$ is smooth, we see that
$X_{c} \backslash \mathcal{H}$ is smooth, so 
our first assertion follows from Theorem \ref{thm:eulerchar}.
The second assertion is a consequence of the description of
the likelihood correspondence $\mathcal{L}_{X_{ c}}$ via
 linear sections of $X_{{ c}}$ that is given in Proposition 
\ref{prop:toricLC} below.
\end{proof}

\begin{example} 
\label{ex:cubicsurface2}
\rm We return to the cubic surface $X_{c}$
in Example \ref{ex:cubicsurface}. For a general parameter vector
${ c}$, the ML degree of $X_{{ c}}$ is $3$.
For instance, the surface $V(p_0p_1p_2-p_3^3) \subset \PP^3$ has ML degree $3$.
 However, the ML degree of $X_{{ c}}$ drops to $2$ whenever the plane curve defined by
\[
 f(x_1,x_2)  = c_0+c_1x_1^3+c_2x_2^3+c_3x_1x_2 \,
 \]
has a singularity in $(\mathbb{C}^*)^2$.
For instance, this happens for $c = (1:1:1:-3)$.
The corresponding surface $V(27p_0p_1p_2+p_3^3) \subset \PP^3$ has ML degree $2$.
\hfill $ \diamondsuit$
\end{example}

The isomorphism (\ref{eq:hscomp}) has a nice interpretation
in terms of Convex Optimization. Namely, it implies that
maximum likelihood estimation for toric varieties is
equivalent to global minimization of {\em posynomials},
and hence to the most fundamental case of {\em Geometric Programming}. 
We refer to \cite[\S 4.5]{BoydVan} for an introduction to 
posynomials and geometric programming. 

We write $| \,\cdot \,|$ for the 1-norm on $\mathbb{R}^{n+1}$, 
we set ${b} = A{ u}$, and we assume that ${ c} = (c_0,c_1,\ldots,c_n)$
is in $ \mathbb{R}_{>0}^{n+1}$.
Maximum likelihood estimation for toric models is the  problem
\begin{equation}
\label{eq:toricMLE2}
 {\rm Maximize} \,\,\,\, \frac{p^{ u}}{|p|^{|{ u}|}} \,\,\,\,
\hbox{subject to}\,\, \,p \,\in\, X_{ c} \cap \Delta_n .
\end{equation}
Setting $p_i = c_i \cdot { x}^{{ \tilde a_i}}$ as above, this problem becomes
 equivalent to  the geometric program
\begin{equation}
\label{eq:toricMLE3}
 {\rm Minimize}\, \,\,\,\frac{f({ x})^{|{ u}|}} {{ x}^{ b}} \,\,\,\,
\hbox{subject to}\,\, \,{ x} \in \mathbb{R}_{>0}^d .
\end{equation}
By construction, $f({ x})^{|{u}|}/{x}^{ b}$ is a posynomial
whose Newton polytope contains the origin.
Such a posynomial attains a unique global minimum on the open orthant $\mathbb{R}^d_{> 0}$.
This can be seen by convexifying as in \cite[\S 4.5.3]{BoydVan}.
This global minimum of (\ref{eq:toricMLE3}) corresponds to the
 solution of (\ref{eq:toricMLE2}), which exists and is unique by
Birch's Theorem \cite[Theorem 1.10]{PS}.

\begin{example} 
\label{ex:cubicsurface3} \rm
Consider the geometric program
for the surfaces in Example \ref{ex:cubicsurface}, with
$$ A \,=\, 
   \begin{pmatrix}
\,0 &  3  & 0 & 1 \,\\ 
\,0 &  0  & 3 & 1 \,\\
\,1  & 1  & 1  & 1\,
\end{pmatrix} \quad 
\hbox{and} \quad 
{ u} \,=\, (0,0,0,1). $$
The problem  (\ref{eq:toricMLE3}) is to
find the global minimum, over all positive $x=(x_1,x_2)$, of the function
$$ \frac{ f(x_1,x_2)}{x_1x_2} \,\, = \,\, 
c_0 x_1^{-1} x_2^{-1} +c_1x_1^2 x_2^{-1} +c_2 x_1^{-1} x_2^2+c_3 .$$
This is equivalent to
maximizing $p_3/p_+$ subject to 
$\,{ p} \in V(c_3^3 \cdot p_0 p_1 p_2\, -\, c_0c_1c_2 \cdot p_3^3) \cap \Delta_3$.
\hfill $\diamondsuit$
\end{example}

We now describe the toric likelihood correspondence
 $\mathcal{L}_{X_{ c}}$ in $\PP^n \times \PP^n$ associated with the pair $(A,{c})$.
 This is the likelihood correspondence of the toric variety
 $X_{ c} \subset \PP^n$ defined above.

\begin{proposition} 
\label{prop:toricLC}
On the open subset $\,(X_{ c} \backslash \mathcal{H}) \times \PP^n$, 
the toric likelihood correspondence $\mathcal{L}_{X_{ c}}$  is defined 
 by the $2 \times 2$-minors  of  the  $2 \times (d {+} 1) $-matrix
\begin{equation}
\label{eq:povercanduoverc}
 \begin{pmatrix} { p}/{ c} \cdot A^T \\
 {u}/{ c} \cdot A^T \end{pmatrix}. 
 \end{equation}
 Here the notation ${ p}/{ c}$ is as in {\rm (\ref{eq:poverc})}.
  In particular, for any fixed data vector ${ u}$,
the critical points of $\ell_{ u}$ are characterized by
a linear system of equations in ${ p}$ restricted to $X_{c}$.
\end{proposition}

\begin{proof}
This is an immediate consequence of Birch's Theorem \cite[Theorem 1.10]{PS}.
\end{proof}

\begin{example} \rm
The  {\em Hardy-Weinberg curve} of Example \ref{ex:hardyweinberg}
is the subvariety $\,X_{c} = V(p_1^2 - 4 p_0 p_2)\,$
in the projective plane $\PP^2$. As a toric variety, this plane curve is given by
$$ A = \begin{pmatrix} 0 & 1 & 2 \\ 2  & 1  & 0 \end{pmatrix}  \quad
\hbox{and}  \quad \,
{ c} = (1,2,1). $$
The likelihood correspondence of $X_{ c}$ is the surface in $\PP^2 \times \PP^2$
given by
\begin{equation}
\label{eq:detdet0}
{\rm det}
\begin{pmatrix}
2 p_0 & p_1 \\
p_1 & 2 p_2 
\end{pmatrix} \,\, = \,\,
{\rm det}
\begin{pmatrix}
p_1 + 2 p_2 &  2 p_0 + p_1 \\
 u_1 + 2 u_2 
 & 2 u_0 + u_1 
\end{pmatrix} 
\,\, = \,\, 0.
\end{equation}
Note that the second determinant equals the determinant
of the $2 \times 2$-matrix (\ref{eq:povercanduoverc}) times $4$.
Saturating (\ref{eq:detdet0}) with respect to~$p_0+p_1+p_2$ 
reveals two further equations of degree $(1,1)$:
$$ 
2 (u_1+2u_2)p_0 =  (2u_0+u_1)p_1 \quad \hbox{and} \quad
  (u_1+2u_2)p_1 =	2 (2u_0+u_1)p_2 .
  $$
For fixed ${u}$, these equations have a unique solution in $\PP^2$, 
 given by the formula in (\ref{eq:hardyweinbergMLE}).
\hfill $ \diamondsuit $          \end{example}

\smallskip

Toric varieties are rational varieties that are parametrized by monomials.
We now examine those varieties that are parametrized by generic polynomials.
Understanding these is useful for statistics since
many widely used models for discrete data are given in the form
$$ f : \Theta \rightarrow \Delta_n , $$
where $\Theta$ is a $d$-dimensional polytope
and $f$ is a polynomial map.
The coordinates $f_0,f_1,\ldots,f_n$ are polynomial functions
in the parameters $\theta = (\theta_1,\ldots,\theta_d)$
satisfying $\,f_0 + f_1 + \cdots + f_n = 1$.
Such models include the mixture
models in Proposition \ref{prop:mix},
phylogenetic models,  Bayesian networks, hidden Markov models, 
and many others arising in computational biology~\cite{PS}.

The model specified by the polynomials $f_0, \ldots,f_n$ is the semialgebraic set
 $f(\Theta) \subset \Delta_n$. We study its Zariski closure 
$\,X = \overline{f(\Theta)}\,$ in $\PP^n$.
Finding its equations is hard and interesting.

\begin{theorem} \label{thm:genericmap}
Let $f_0,f_1,\ldots,f_n$ be polynomials of degrees $b_0,b_1,\ldots,b_n$
satisfying $\sum f_i = 1$.
The ML degree of the variety $X$ is at most the coefficient of $z^d$ in
the generating function
$$ \frac{(1-z)^d}{(1-zb_0)(1-zb_1) \cdots (1-zb_n)} .$$
Equality holds when the coefficients of $f_0,f_1,\ldots,f_n$ are generic
relative to $\sum f_i = 1$.
\end{theorem}

\begin{proof}
This is the content of \cite[Theorem 1]{CHKS}.
\end{proof}

\begin{example} \rm
We examine the case of quartic surfaces in $\PP^3$.
Let $d=2,n=3$, pick random affine quadrics
$f_1,f_2,f_3$ in  two unknowns
and set $f_0  = 1-f_1-f_2-f_3$. This defines a map
$$ f \,:\,\mathbb{C}^2 \rightarrow  \mathbb{C}^3 \subset \mathbb{P}^3. $$
The ML degree of the image surface $X = \overline{f(\mathbb{C}^2)}$
in $\PP^3$ is equal to $25$ since
$$ \frac{(1-z)^2}{(1-2z)^4} \, = \, 1  + 6z + {\bf 25} z^2 + 88 z^3 + \cdots $$
The rational surface $X$ is a Steiner surface (or Roman surface).
Its singular locus consists of three lines that meet in a point $P$.
To understand the graph of  $f$, we observe that the linear span
of $\{f_0,f_1,f_2,f_3\}$ in $\mathbb{C}[x,y]$ has a basis $\{1,L^2,M^2,N^2\}$ where 
$L,M,N$ represent lines in $\mathbb{C}^2$. Let
$ l$ denote the line through $M \cap N $ parallel to $L$, 
 $m$ the line through $L \cap N$ parallel to  $M$, and 
  $n$ the 
line through $L \cap M$ parallel to $N$. 
The map $ \mathbb{C}^2 \rightarrow X$  is a bijection outside these 
three lines, and it maps each line 2-to-1 onto one of the lines in $X_{\rm sing}$.
The fiber over the special point $P$ on $X$ consists of 
three points, namely, $l \cap m$, $ l \cap n$ and $m \cap n$.
If the quadric $f_0$ were also picked at random, rather than 
as $1-f_1-f_2-f_3$, then we would still get a 
Steiner surface $X \subset \PP^3$. However,
 now the  ML degree of $X$ increases to $33$.

On the other hand, if we take $X$ to be a general 
quartic surface in $\PP^3$, so $X$ is a smooth K3 surface of Picard 
rank $1$, then $X$ has ML degree $84$. This is the formula in
Example \ref{ex:generichyper} evaluated at
$n=3$ and $d=4$. Here  $X \backslash \mathcal{H}$ 
is the generic quartic surface in $\PP^3$ with five plane sections removed.
The number $84$ is the Euler characteristic of that
open K3 surface.

In the first case, $X \backslash \mathcal{H}$ is singular,
so we cannot apply Theorem \ref{thm:eulerchar} directly
 to our Steiner surface $X$ in $\PP^3$.
However, we can work in the parameter space and
consider the smooth very affine surface $\mathbb{C}^2 \backslash V(f_0f_1f_2f_3)$.
The number $25$ is the Euler characteristic of that surface.

It is instructive to verify  Conjecture \ref{conj:eulerpositive} for our
three quartic surfaces in $\PP^3$. We found
$$
\begin{matrix}
\chi(X \backslash \mathcal{H}) &=& 38 && > && 25 &=& {\rm MLdegree}\,(X), \\
\chi(X \backslash \mathcal{H}) &=& 49 && > && 33 &=& {\rm MLdegree}\,(X), \\
\chi(X \backslash \mathcal{H}) &=& 84 && = && 84 &=& {\rm MLdegree}\,(X).
\end{matrix}
$$
The Euler characteristics of the three surfaces were computed using Aluffi's method 
 \cite{AluJSC}.
\hfill $\diamondsuit$
\end{example}

\smallskip

We now turn to the following question:~{\em which projective varieties $X$ have ML degree one?}
This question is important for likelihood inference
because a model having ML degree one means that the MLE $\hat p$ is a 
rational function in the data $u$.
It is known that Bayesian networks and decomposable graphical
models enjoy this property, and it is natural to wonder which
other statistical models are in this class.
The answer to this question was given by the first author in \cite{Huh2}.
We shall here present the result of \cite{Huh2} from a slightly different angle.

Our point of departure is the notion of the $A$-discriminant,
as introduced and studied by Gel'fand, Kapranov and Zelevinsky in \cite{GKZ}.
We fix an $r \times m$ integer matrix
$A = (a_1,a_2,\ldots,a_m)$  
of rank $r$  which   has $(1,1,\ldots,1)$ in its row space. The 
Zariski closure of 
$$ \bigl\{ ({t}^{a_1}: { t}^{a_2} : \cdots : 
{t}^{a_m} )\,\in \mathbb{P}^{m-1} \,\,: \,\, { t} \in (\mathbb{C}^*)^r \bigr\} $$
is an $(r-1)$-dimensional  toric variety $Y_A$ in $\mathbb{P}^{m-1}$. 
We here intentionally changed the notation relative to that used 
for toric varieties at the beginning of this section.
The reason is that $d$  and $n$ are always reserved
for the dimension and embedding dimension of a statistical model.

The {\em dual variety} $\,Y_A^*\,$   is  
an irreducible variety in the dual projective space
$(\mathbb{P}^{m-1})^\vee$ whose coordinates are ${ x} = (x_1:x_2: \cdots : x_m)$.
We identify points ${ x}$ 
in $(\mathbb{P}^{m-1})^\vee$ with  hypersurfaces
\begin{equation}
\label{eq:xhyper}
\bigl\{ \, { t} \in (\mathbb{C}^*)^r \,\,:\,\,
x_1 \cdot { t}^{a_1} + 
x_2 \cdot { t}^{a_2} + \cdots + 
x_m \cdot { t}^{a_m}  \,\, = \,\, 0 \, \bigr\}.
\end{equation}
The dual variety $Y_A^*$ is
the Zariski closure in $(\mathbb{P}^{m-1})^\vee$ of 
the locus of all hypersurfaces (\ref{eq:xhyper}) that are singular.
Typically, $Y_A^*$ is a hypersurface.
In that case, $Y_A^*$ is defined by a unique (up to sign) 
irreducible polynomial
${ \Delta_A} \in \mathbb{Z}[x_1,x_2,\ldots,x_m]$.
The homogeneous polynomial $\,{ \Delta_A} \,$ is called the  {\em $A$-discriminant}.
Many classical discriminants and resultants are instances  of $\Delta_A$.
So  are determinants and hyperdeterminants. This is the punch line of the book \cite{GKZ}.

\begin{example} \label{ex:disccubic} \rm
Let $ m = 4, r = 2$, and $A = \begin{pmatrix} 3 & 2 & 1 & 0 \\ 0 & 1 & 2 & 3 \end{pmatrix} $. 
The associated toric variety  is the twisted 
 cubic curve 
 \[
 Y_A=\overline{\bigl\{(1: t : t^2 :  t^3)\,|\, t \in \mathbb{C} \bigr\}}\, \subset \mathbb{P}^3.
 \] 
The variety $Y_A^*$  that is dual to the curve $Y_A$
is a surface in $(\mathbb{P}^3)^\vee$. The surface $Y_A^*$ parametrizes
all planes that are tangent to the curve $Y_A$.
These represent  univariate cubics
$$ x_1 + x_2  t   + x_3 t^2 + x_4 t^3
$$
 that have a double root.
Here the  $A$-discriminant is the classical discriminant
$$ \Delta_A \,= \,27 x_1^2 x_4^2
-18 x_1 x_2 x_3 x_4+4 x_1 x_3^3+4 x_2^3 x_4-x_2^2 x_3^2.
$$
The surface $Y_A^*$ in $\PP^3$ defined by this equation
is the discriminant of the univariate cubic.
\hfill $ \diamondsuit$
\end{example}

\begin{theorem} \label{thm:huh2}
Let $X \subseteq \PP^n$ be a projective variety of ML degree $1$.
Each coordinate
$\hat p_i$ of the rational function $\,u \mapsto \hat p\,$ is
an alternating product of linear forms in $u_0,u_1,\ldots,u_n$.
\end{theorem}

The paper \cite{Huh2} gives an explicit construction
of the map $u \mapsto \hat p$ as a {\em Horn uniformization}. 
A precursor was \cite{Kapranov}.
We explain this construction. The point of departure is a 
matrix $A$ as above. We now take
$\Delta_A$ to be any non-zero homogenous polynomial
that vanishes on the dual variety $Y_A^*$
of the toric variety $Y_A$. If $Y_A^*$ is a hypersurface
then $\Delta_A$ is the $A$-discriminant.

First, we write $\Delta_A$
as a Laurent polynomial by dividing it by
one of its monomials:
\begin{equation}
\label{eq:scaleddisc}
 \frac{1}{\rm monomial} \cdot \Delta_A \,\,\, = \,\,\,
1 \,-\, c_0 \cdot { x}^{b_0} \,-\, c_1 \cdot { x}^{b_1}
\, -\, \cdots \,-\, c_n \cdot { x}^{b_n}. 
\end{equation}
This expression defines an $m \times (n+1)$  integer matrix $B=(b_0,\ldots,b_n)$
satisfying $AB = 0$.
Second, we define $X$ to be the rational subvariety of $\PP^n$ 
that is given parametrically by
\begin{equation}
\label{eq:hornnn}
 \frac{p_i}{p_0 + p_1 + \cdots + p_n} \,\,\,= \,\,\,c_i \cdot { x}^{b_i} \qquad
\hbox{for} \,\, \, i = 0,1,\ldots,n .
\end{equation}
The defining ideal of $X$ is obtained by eliminating $x_1,\ldots,x_m$
from the equations above. Then $X$ has ML degree $1$,
and, by \cite{Huh2},
every variety of ML degree $1$ arises in this manner.

\begin{example} \label{ex:disccubic2} \rm
The following curve in $\PP^3$ happens to be a variety of ML degree $1$:
$$ X \,\,= \,\, V \bigl(\, 9 p_1 p_2 - 8 p_0 p_3\,,\,
p_0^2-12 (p_0{+}p_1{+}p_2{+}p_3) p_3 \, \bigr). $$
This curve comes from the discriminant of the univariate cubic in
Example \ref{ex:disccubic}:
$$ \frac{1}{{\rm monomial}}\cdot \Delta_A \,\,= \,\,
1 -  \bigl(\frac{2}{3} \frac{ x_2 x_3} {x_1 x_4}\bigr)
- \bigl( -\frac{4}{27} \frac{x_2^3}{x_1^2 x_4} \bigr)
- \bigl( - \frac{4}{27} \frac{x_3^3}{x_1 x_4^2} \bigr)
- \bigl( \frac{1}{27} \frac{x_2^2 x_3^2}{x_1^2 x_4^2} \bigr) . $$
We derived the curve $X$ from the four parenthesized monomials 
via the formula (\ref{eq:hornnn}).
The maximum likelihood estimate for this model is given by the products of linear forms
$$
\hat p_0  = \frac{2}{3} \frac{ x_2 x_3} {x_1 x_4} \qquad
\hat p_1 =  -\frac{4}{27} \frac{x_2^3}{x_1^2 x_4}  \qquad
\hat p_2 =  - \frac{4}{27} \frac{x_3^3}{x_1 x_4^2}  \qquad
\hat p_3 =  \frac{1}{27} \frac{x_2^2 x_3^2}{x_1^2 x_4^2}  $$
where
$$
\begin{matrix}
  x_1  &=& -u_0 - u_1 - 2 u_2 - 2 u_3 && 
  x_2 &=& u_0 + 3 u_2 + 2 u_3 \\ 
   x_3 &=& u_0 + 3 u_1 + 2 u_3 &&
    x_4 &=&  -u_0 - 2 u_1 - u_2 - 2 u_3
    \end{matrix}
    $$
    These expressions  are 
   the alternating products of linear forms promised 
   in Theorem \ref{thm:huh2}. 
\hfill $\diamondsuit$
\end{example}

We now give the formula for $\hat p_i$
in general. This is the 
{\em Horn uniformization} of \cite[\S 9.3]{GKZ}.

\begin{corollary} \label{cor:horn}
Let $X \subset \PP^n$  be the variety of ML degree $1$ 
with parametrization (\ref{eq:hornnn})
derived from a scaled $A$-discriminant
(\ref{eq:scaleddisc}).
The coordinates of the MLE function $u \mapsto \hat p$ are
$$
\hat p_k \quad = \quad c_k \cdot \prod_{j=1}^m
(\sum_{i=0}^n b_{ij} u_i )^{b_{kj}}.
$$
\end{corollary}

It is not obvious (but true)
 that $\,\hat p_0 + \hat p_1 + \cdots + \hat p_n = 1 \,$
holds in the formula above. In light of its monomial
parametrization, our variety $X$ is 
 toric in $\PP^n \backslash \mathcal{H}$.
In general, it is not toric in $\PP^n$, due to appearances
of the factor $(p_0+p_1+\cdots+p_n)$ in equations for $X$. 
Interestingly, there
are numerous instances when this factor does not appear
and $X$ is toric also in $\PP^n$.

One toric instance is the independence model 
$\,X = V(p_{00} p_{11} - p_{01} p_{10})$,
whose MLE was derived in
Example \ref{ex:biologist}. What is the matrix $A$ in this case?
We shall answer this question for a slightly larger example,
which serves as an illustration for
 {\em decomposable graphical models}.

\begin{example} \rm
Consider the conditional independence model 
  for three binary variables given by the graph
\ { $\bullet$-----$\bullet$-----$\bullet$}.
We claim that this graphical model is derived from
$$ A \,\,\,=\,\, \bordermatrix{
        & a_{00} & a_{10} & a_{01} & a_{11} & b_{00} & b_{01} & b_{10} & b_{11} & c_0 & c_1 & d \cr
        &   1      &    1       &     1       &    1       &    1       &    1       &    1       &    1       &    1       &    1   & 1 \cr
   x  \! &   1      &    1       &     0       &    0       &    0       &    0       &    0       &    0       &    1       &    0   & 0 \cr
  y   \! &   0      &    0       &     1       &    1       &    0       &    0       &    0       &    0       &    0       &    1   & 0 \cr
  z   \! &   0     &    0       &     0       &    0       &    1       &    1       &    0       &    0       &    1      &    0   & 0 \cr
  w  \! &   0     &    0       &     0       &    0       &    0       &    0       &    1       &    1       &    0       &    1   & 0 \cr}.
$$
The discriminant of the corresponding family of hypersurfaces
$$ 
\bigl\{ (x,y,z,w) \in (\mathbb{C}^*)^4 \,| \,
(a_{00}+a_{10}) x + (a_{01}+a_{11}) y +
(b_{00}+b_{01}) z + (b_{10}+b_{11}) w + 
c_0 x z + c_1 yw + d = 0 \bigr\} $$
equals
$$ \begin{matrix}
\Delta_A & =  & 
c_0 c_1 d - 
a_{01} b_{10} c_0 - a_{11} b_{10} c_0 -a_{01} b_{11} c_0
-a_{11} b_{11} c_0 \\ & & - a_{00} b_{00} c_1 - a_{10} b_{00} c_1
- a_{00} b_{01} c_1 - a_{10} b_{01} c_1.
\end{matrix}
$$
We divide this $A$-discriminant by its first term $\,c_0 c_1 d\,$
to rewrite it in the form (\ref{eq:scaleddisc}) with $n= 7$.
The parametrization of $X \subset \PP^7$
given by (\ref{eq:hornnn})  can be expressed as
\begin{equation}
\label{eq:twochain}
 \quad p_{ijk} \, = \, \frac{a_{ij} \cdot b_{jk}}{c_j \cdot d} \quad
\qquad \hbox{for} \,\,i,j,k \in \{0,1\}. 
\end{equation}
This is indeed the desired graphical model
$\, \hbox{$\bullet$-----$\bullet$-----$\bullet$} \,$
with implicit representation 
$$
X \,\, = \,\, V \bigl(
p_{000} p_{101} -p_{001} p_{100}\,,\,
p_{010} p_{111} -p_{011} p_{110} \bigr) \,\,\, \subset \,\,\, \PP^7.
$$
The linear forms used in the Horn uniformization of Corollary \ref{cor:horn} are
$$
a_{ij} = u_{ij+} \qquad
b_{jk} = u_{+jk} \qquad
c_j = u_{+j+} \qquad
d = u_{+++}
$$
Substituting these expressions into (\ref{eq:twochain}), we obtain
$$ \qquad \hat p_{ijk} \,\, = \,\,
\frac{u_{ij+}\cdot u_{+jk}}{u_{+j+}\cdot u_{+++}} 
\qquad \hbox{for} \,\,i,j,k \in \{0,1\}. $$
This is the formula in Lauritzen's book \cite{Lau}
for MLE of decomposable graphical models.
\hfill $\diamondsuit$
\end{example}

We now return to the likelihood geometry of an arbitrary
$d$-dimensional projective variety $X$ in $\PP^n$,
as always defined over $\mathbb{R}$ and not contained in $\mathcal{H}$.
We define the  {\em ML bidegree} of $X$ to be the bidegree
of its likelihood correspondence $\mathcal{L}_X  \subset \PP^n \times \PP^n$.
This is a binary form
$$ B_X(p,u) \quad = \quad
(b_0 \cdot p^d + b_1 \cdot p^{d-1}u+ \cdots + b_d \cdot u^d) \cdot p^{n-d},
$$
where $b_0,b_1,\ldots,b_d$ are certain positive integers.
By definition, $B_X(p,u)$ is the multidegree \cite[\S 8.5]{MS}
of the prime ideal of $\mathcal{L}_X$, with respect to the
natural $\mathbb{Z}^2$-grading on 
the polynomial ring $\R [p,u] = \R[p_0,\ldots,p_n, u_0,\ldots, u_n]$.
Equivalently, the ML bidegree $B_X(p,u)$ is the class defined by
$\mathcal{L}_X$ in the cohomology ring
\[
H^*(\mathbb{P}^n \times \mathbb{P}^n;\mathbb{Z}) =                           
\mathbb{Z}[p,u]/\langle p^{n+1}, u^{n+1} \rangle.
\]
We already saw some examples,  for the Grassmannian $G(2,4)$ in 
(\ref{eq:GrassmannBidegree}),
for arbitrary linear spaces in (\ref{eq:linearBidegree}), and for
a toric model of ML degree $1$ in (\ref{eq:BSpreview}).
We note that the bidegree $B_X(p,u)$ can be computed
conveniently using the command
{\tt multidegree} in {\tt Macaulay2}.

To understand the geometric meaning of the ML bidegree,
we introduce a second polynomial.
Let $L_{n-i}$ be a sufficiently general linear subspace of $\mathbb{P}^n$ of codimension $i$, and define
\[
s_i\,=\,\text{MLdegree}\,(X \cap L_{n-i}).
\]
We define the \emph{sectional ML degree} of $X$ to be the polynomial
$$ S_X(p,u)  \quad = \quad
(s_0 \cdot p^d + s_1 \cdot p^{d-1}u+ \cdots + s_d \cdot u^d) \cdot p^{n-d},$$

\begin{example} \rm
The sectional ML degree of the Grassmannian 
$G(2,4 )$ in (\ref{eq:pluckerquadric})  equals
$$ S_X(p,u) \,\,=\,\,4  p^5+20 p^4u+24 p^3u^2+12 p^2u^3+2 pu^4. $$
Thus, if $H_1,H_2,H_3$ denote generic hyperplanes in $\PP^5$,
then the threefold $G(2,4) \cap H_1$ has ML degree $20$,
the surface $G(2,4) \cap H_1 \cap H_2$ has ML degree $24$,
and the curve $G(2,4) \cap H_1 \cap H_2 \cap H_3$
has ML degree $12$. Lastly, the coefficient $2$ of $pu^4$ is simply the degree of $G(2,4)$ in $\PP^5$.
\hfill $\diamondsuit$
\end{example}

\begin{conjecture}
\label{conj:aluffi}
The ML bidegree and the sectional ML degree
of any projective variety $X \subset \PP^n$, not lying in $\mathcal{H}$, are related
by the following involution on binary forms of degree~$n$:
\begin{eqnarray*}
B_X(p,u) &=& \frac{u \cdot S_X(p,u-p)-p \cdot S_X(p,0)}{u-p},\\
S_X(p,u) &=& \frac{u \cdot B_X(p,u+p)+p \cdot B_X(p,0)}{u+p}.
\end{eqnarray*}
\end{conjecture}

This conjecture is a theorem when $X \backslash \mathcal{H}$
is smooth and its boundary is {\em sch\"on}. See Theorem \ref{proCSM2} below.
In that case, the ML bidegree is identified,
by \cite[Theorem 2]{Huh1}, with the
Chern-Schwartz-MacPherson (CSM) class of the
constructible function on $\PP^n$ that is $1$ on
$X \backslash \mathcal{H}$ and $0$ elsewhere.
Aluffi proved in \cite[Theorem 1.1]{Alu}
that the CSM class of an locally closed subset of $\PP^n$
satisfies such a  {\em log-adjunction formula}.
Our formula in Conjecture \ref{conj:aluffi} is precisely the homogenization
of Aluffi's involution.
The combination of  \cite[Theorem 1.1]{Alu} and
\cite[Theorem 2]{Huh1} proves Conjecture \ref{conj:aluffi} in
cases such as generic complete intersections (Theorem \ref{thm:genericMLdegree})
and arbitrary linear spaces (Theorem \ref{LinearCase}).
In the latter case, it can also be verified using matroid theory.
 Conjecture \ref{conj:aluffi} says that this holds for any $X$,
indicating a deeper connection between likelihood correspondences and CSM classes.

We note that  $B_X(p,u)$ and $S_X(p,u)$ always share
the same leading term and the same trailing term,
and this is compatible with our formulas. Both polynomials
start and end like
$$ {\rm MLdegree}\,(X) \cdot p^n \,+ \,\cdots \,+ {\rm degree}\,(X) \cdot p^{{\rm codim}(X)} u^{{\rm dim}(X)}.
$$
We now illustrate Conjecture \ref{conj:aluffi} by
verifying it  computationally for a few more examples.

\begin{example} \rm
Let us examine some cubic fourfolds in $\PP^5$.
If $X$ is a generic hypersurface of degree $4$ in $\PP^5$
then its sectional ML degree and ML bidegree satisfy the conjectured formula:
$$ \begin{matrix}
S_X(p,u) &=&
 1364 p^5  \,+\,  448 p^4 u \, +\,  136 p^3 u^2\,+\, 32 p^2 u^3 \,+\, 3 p u^4 ,\\
 B_X(p,u) & = &
 1364 p^5  \,+\,  341 p^4 u \, +\,  81 p^3 u^2\,+\, 23 p^2 u^3 \,+\, 3 p u^4 .
 \end{matrix}
 $$
Of course, in algebraic statistics, we are more interested in 
special hypersurfaces that are statistically meaningful.
One such instance was seen in  Example \ref{ex:galois33}.
The mixture model for two identically distributed 
ternary random variables is the fourfold $X \subset \PP^5$ defined by
\begin{equation}
\label{eq:3x3symmetric}
{\rm det} \begin{pmatrix}                                                 
2p_{11} & p_{12} & p_{13}  \\
p_{12} & 2 p_{22} & p_{23} \\
p_{13} & p_{23} & 2 p_{33} 
\end{pmatrix} \quad = \quad 0 .
\end{equation}
The sectional ML degree and the ML bidegree of this determinantal fourfold are
$$ \begin{matrix}
S_X(p,u) &=& 6 p^5 \,+\,  42p^4 u \,+\,   48 p^3 u^2 \,+\,  21 p^2 u^3  \,+\, 3 p u^4  \\
 B_X(p,u) & = &
6 p^5  \,+\,   12 p^4 u \, +\,  15 p^3 u^2\,+\,  12 p^2 u^3 \,+\, 3 p u^4 .
 \end{matrix}
 $$
 For  the toric fourfold $ X = V(p_{11}p_{22}p_{33}\,{ -}\,p_{12}p_{13}p_{23})$,
ML bidegree and sectional ML degree~are 
$$
\begin{matrix}
B_X(p,u) &=&{\bf 3} p^5  \,+\,  3 p^4 u \, +\,  3 p^3 u^2\,+\,3 p^2 u^3 \,+\, 3 p u^4, \\
S_X(p,u) &=& {\bf 3} p^5  + 12 p^4 u  +  18 p^3 u^2+12 p^2 u^3 + 3 p u^4. 
\end{matrix}
$$
Now, taking $ X = V(p_{11}p_{22}p_{33} { +} p_{12}p_{13}p_{23})$ instead,
the leading coefficient $3$ changes to $2$. 
\hfill $\diamondsuit$
\end{example}

 \begin{remark}  \rm
  Conjecture \ref{conj:aluffi} is true when $X_{ c}$ is a toric variety with ${ c}$
 generic, as in Theorem \ref{thm:toricMLE}.
 Here we can use Proposition \ref{prop:toricLC} to infer
 that all coefficients of $B_X$ are equal to 
the normalized volume of the lattice polytope ${\rm conv}(A)$. In symbols, 
for generic ${ c}$, we have
$$  B_{X_{ c}}(p,u) \quad = \quad
   {\rm degree}\,(X_{ c}) \cdot \sum_{i=0}^d   p^{n-i} u^i . $$
It is now an exercise to transform this into a formula for the sectional
ML degree $ S_{X_{ c}}(p,u)$.
 \end{remark}
 
 In general, it is hard to compute generators for the
 ideal of the likelihood correspondence.
 
 \begin{example}\label{ex:secant} \rm
The following submodel of (\ref{eq:3x3symmetric})
was featured prominently in \cite[\S 1]{HKS}:
 \begin{equation}
\label{eq:3x3hankel}
{\rm det} \begin{pmatrix}                                                 
12 p_0 & 3 p_1 & 2 p_2  \\
3 p_1 & 2 p_2 & 3 p_3 \\
2 p_2  & 3 p_3 & 12 p_4 
\end{pmatrix} \quad = \quad 0 .
\end{equation}
This cubic threefold $X$ is the secant variety of a rational normal curve in $\PP^4$,
and it represents the mixture model for a binomial random variable
(tossing a biased coin four times). It takes several hours in {\tt Macaulay2}
to compute the prime ideal of the likelihood correspondence
$\mathcal{L}_X \subset \PP^4 \times \PP^4$. That ideal has
$20$ minimal generators
one in degree $(1,1)$,
one in degree $(3,0)$,
five in degree $(3,1)$,
ten in degree $(4,1)$ and
three in degree $(3,2)$.
After passing to a  Gr\"obner basis,
we use the formula in \cite[Definition 8.45]{MS}
to compute the
bidegree of $\mathcal{L}_X$:
$$ B_X(p,u) \quad = \quad 
12 p^4 + 15 p^3 u + 12 p^2 u^2 + 3 p u^3. 
$$
We now intersect $X$ with
 random hyperplanes in $\PP^4$,
and we compute the  ML degrees of the intersections.
Repeating this experiment many times
reveals the sectional ML degree of $X$:
$$ S_X(p,u) \quad = \quad 
12 p^4 + 30 p^3 u + 18 p^2 u^2 + 3 p u^3. 
$$
The two polynomials satisfy our
transformation rule, thus confirming Conjecture \ref{conj:aluffi}.
We note that Conjecture \ref{conj:eulerpositive} also holds for this example:
using Aluffi's method \cite{AluJSC}, we find $\chi( X \backslash \mathcal{H}) = -13$.
\hfill $\diamondsuit$
 \end{example}
 
 \smallskip
 
Our last topic is the operation of restriction and deletion.
This is a standard tool for  complements of hyperplane 
arrangements, as in Theorem \ref{LinearCase}. It was 
developed  in \cite{Huh1}  for arbitrary very affine varieties, 
such as $X \backslash \mathcal{H}$.
We motivate this by explaining the distinction between
{\em structural zeros} and {\em sampling zeros} for
contingency tables in statistics \cite[\S 5.1.1]{BFH}.

Returning to the ``hair loss due to TV soccer'' example from the beginning of Section 2,
let us consider the following questions. What is the difference between the data set
$$ U \quad \, = \,\,\,\,
\bordermatrix{
 &  \hbox{lots of hair} & \hbox{medium hair} & \hbox{little hair} \cr
\hbox{$\leq 2$ hrs} & 15 & 0 & 9 \cr
\hbox{$2$--$6$ hrs} &  20 & 24 & 12 \cr
\hbox{$\geq 6$ hrs} & 10 & 12 &  6 }
$$
and the data set
$$ \tilde U \quad \, = \,\,\,\,
\bordermatrix{
 &  \hbox{lots of hair} & \hbox{medium hair} & \hbox{little hair} \cr
\hbox{$\leq 2$ hrs} & 10 & 0 & 5 \cr
\hbox{$2$--$6$ hrs} &  9 & 3 & 6 \cr
\hbox{$\geq 6$ hrs} &  7 & 9 &  8} \ ?
$$
How should we think about the zero entries in row 1 and column 2
of these two contingency tables?
Would the rank 1 model $\mathcal{M}_1$ or the rank 2 model
$\mathcal{M}_2$ be more appropriate?

The first matrix $U$ has rank $2$ and it can be completed to a rank $1$ matrix
by replacing the zero entry with $18$. Thus, the model $\mathcal{M}_1$
fits perfectly except for  the  {\em structural zero}
in row 1 and column 2.  It seems that this zero is
inherent in the structure of the problem:  planet Earth
simply has no people with medium hair length who rarely watch
soccer on TV.

The second matrix $\tilde U$ also has rank two, but it cannot 
be completed to rank $1$. The model $\mathcal{M}_2$
is a perfect fit. The zero entry in $\tilde U$ appeared to be
an artifact of the particular group that was interviewed in this study.
This is a {\em sampling zero}. It arose because, by chance,
in this cohort nobody happened to have medium hair length and
watch soccer on TV rarely. We refer to
the book of Bishop, Feinberg and Holland \cite[Chapter 5]{BFH}
for an introduction.

We now consider an arbitrary projective variety $X \subseteq \PP^n$, 
serving as our statistical model. Suppose that
structural zeros or sampling zeros occur in the last coordinate $u_n$.
Following \cite[Theorem 4]{Rapallo}, we model
structural zeros by the projection $\pi_n(X) $. This model is
the variety in $\PP^{n-1}$ that is the closure of the image of $X$ under the rational map
\[
\pi_n:\PP^n \dashrightarrow \PP^{n-1}, \qquad (p_0:p_1:\cdots:p_{n-1}:p_n) \longmapsto (p_0:p_1:\cdots:p_{n-1}).
\]
Which projective variety is a good representation for sampling zeros?
We propose that sampling zeros be modeled
by the  intersection $X \cap \{p_n{=}0\}$. This is now
to be  regarded as a subvariety in $\PP^{n-1}$.  
In this manner, both structural zeros and sampling zeros
are  modeled by closed subvarieties of $\PP^{n-1}$.
Inside that ambient $\PP^{n-1}$, our
 standard arrangement $\mathcal{H}$ consists of $n+1$ hyperplanes.
Usually, none of these  hyperplanes contains $X \cap \{p_n{=} 0\}$ or $\pi_n(X)$.

It would be desirable to express the (sectional) ML degree of $X$  in terms 
of those of the intersection $X \cap \{p_n = 0\}$ and the projection $\pi_n(X)$. 
As an alternative to  the ML degree of the projection $\pi_n(X)$ into $\PP^{n-1}$,
here is a quantity in $\PP^n$ that reflects the presence of
structural zeros even more accurately.
We denote by
$$ {\rm MLdegree}\,(X|_{u_n = 0}) $$
the number of critical points
$\,\hat p = (\hat p_0: \hat p_1: \cdots :\hat p_{n-1}:\hat p_n)\,$
of $\,\ell_u\,$ in $\,X_{\rm reg} \backslash \mathcal{H}\,$
for those data vectors $\,u = (u_0,u_1,\ldots,u_{n-1},0)\,$
whose first $n$ coordinates $u_i$ are positive and generic.

\begin{conjecture} \label{conj:3.19}
The maximum likelihood degree
satisfies the inductive formula
\begin{equation}
\label{eq:MLDinductive}
{\rm MLdegree}\,(X) \,\, = \,\,
{\rm MLdegree}\,(X \cap \{p_n{=} 0\}) \,+\,
{\rm MLdegree}\,(X|_{u_n=0}) ,
\end{equation}
provided  $X$ and $X \cap \{p_n{=}0\}$ are reduced,
irreducible,
and not contained in their respective~$\mathcal{H}$.
\end{conjecture}

We expect that an analogous formula will hold for the sectional 
ML degree $S_X(p,u)$.
The intuition behind equation (\ref{eq:MLDinductive}) is as 
follows. As the 
data vector  $u$ moves from a general point in $\PP^n_u$
to a general point on the hyperplane $\{u_n = 0\}$, the 
corresponding fiber
${\rm pr}_2^{-1}(u)$ of the likelihood fibration
splits into two clusters. One cluster has size
${\rm MLdegree}\,(X|_{u_n=0}) $ and  stays away from $\mathcal{H}$.
The other cluster moves onto the hyperplane $\{p_n = 0\}$
in $\PP^n_p$, where it approaches the various critical 
points of $\ell_u$  in that intersection. This degeneration 
is the perfect scenario for
a numerical homotopy, e.g.~in {\tt Bertini}, as discussed in 
Section 2. These homotopies are currently being studied for 
determinantal varieties  by
Elizabeth Gross and Jose Rodriguez \cite{GR}.
The formula (\ref{eq:MLDinductive}) has been verified 
computationally for many examples. Also,  Conjecture  \ref{conj:3.19}
is known to be true in the slightly different setting
of \cite{Huh1}, under a certain smoothness assumption.
This is the content of \cite[Corollary 3.2]{Huh1}.

\smallskip


\begin{example} \rm
Fix the space $\PP^8$ of $3 \times 3$-matrices
as in \S 2.
For the rank $2$ variety  $X = \mathcal{V}_2$,
the formula (\ref{eq:MLDinductive}) reads
$\,10=5+5$. For the rank $1$ variety $X = \mathcal{V}_1$,
it reads  $\,1 = 0+1$.
\hfill $\diamondsuit$
\end{example}

\begin{example} \rm
If $X$ is a
generic $(d,e)$-curve in $\PP^3$, then
\[
{\rm MLdegree}\,(X) = d^2e + de^2 + de \quad \text{and} \quad X \cap \{p_3 = 0\}=(\text{$d \cdot e$ distinct points}).
\]
Computations suggest that
\[
{\rm MLdegree}\,(X|_{u_3 = 0})  =   d^2e + de^2 \quad \text{and} \quad
{\rm MLdegree}\,(\pi_3(X))  =   d^2e + de^2 .  
\]
 To derive the second equality geometrically, one may argue as follows.
 Both curves $X \subset \PP^3$ and  $\pi_3(X) \subset \PP^2$
 have degree $de$ and genus $ \frac{1}{2}(d^2 e + d e^2) - 2 de + 1$.
 Subtracting this from the expected genus
 $\frac{1}{2}(de-1)(de-2)$ of a plane curve of degree $de$,
 we find that $\pi_3(X)$ has $ \frac{1}{2}d (d-1)e(e-1)$ nodes.
 Example \ref{ex:nodecusp} suggests that
each node decreases the ML degree of a plane curve by $2$.
Assuming this to bet the case, we conclude
$${\rm MLdegree}\,(\pi_3(X))\,\,=\,\, de(de+1) -   d(d-1)e(e-1)
\,\,=\,\,  d^2e + de^2.$$ 
Here we are using that a general plane curve
of degree $de$ has ML degree $de(de+1)$.
\hfill $\diamondsuit$
\end{example}

This example suggests that, in favorable circumstances, the following
identity would hold:
\begin{equation}
\label{eq:favorable}
\hbox{MLdegree}\,(X|_{u_n=0})  \, \, = \,\,
\hbox{MLdegree}\,(\pi_n(X)).
\end{equation}
However, this is certainly not true in general. Here is a particularly telling example:

\begin{example} \rm
Suppose that $X$ is a generic surface of degree $d$ in $\PP^3$. Then
$$\begin{matrix}
\hbox{MLdegree}\,(X) &=& d+d^2+d^3 ,\\
\hbox{MLdegree}\,(X \cap \{p_3 = 0\}) & = &  d+d^2 ,\\
\hbox{MLdegree}\,(X|_{u_3 = 0}) & = &  d^3 ,\\
 \hbox{MLdegree}\,(\pi_3(X)) & = & 1.
  \end{matrix}
 $$
 Indeed, for most hypersurfaces  $X \subset \PP^n$,
 the same will happen, since     $\pi_n(X) = \PP^{n-1}$.  \hfill $\diamondsuit$
\end{example}

As a next step, one might conjecture
that (\ref{eq:favorable}) holds
when the map is birational
and  the center  $(0:\cdots:0:1)$ of the projection does not
lie on the variety $X$.
But this also fails:

\begin{example} \rm
Let $X$ be the twisted cubic curve in $\PP^3$ defined
by the $2 \times 2$-minors of 
$$ \begin{pmatrix}
\,           p_0+p_1-p_2 &&  2 p_0-p_2+9p_3  && p_0-6 p_1+8 p_2 \, \\
  \,  2 p_0-p_2+9 p_3 && p_0-6 p_1+8 p_2 && 7 p_0+p_1+2 p_2 \,
    \end{pmatrix} . $$
The ML degree of $X$ is $13 = 3+10 $, and $X$ intersects $\{p_3=0\}$ in 
three distinct points.
 The projection of the curve $X$ into $\PP^2$ is a cuspidal cubic,
 as in Example \ref{ex:nodecusp}. We have
 \[
{\rm MLdegree}\,(X|_{u_3 = 0})  =  10   \quad \text{and} \quad
{\rm MLdegree}\,(\pi_3(X))  =  9.
 \]
 It is also instructive to compare the number 
 $\,13 = -\chi(X \backslash \mathcal{H})\,$ with the
 number $\,11\,$ one gets in Theorem  \ref{thm:genericmap}
 for the special twisted cubic curve with
  $d = 1$, $n=3$ and $b_0=b_1=b_2=b_3 = 3$.
There are many mysteries
 still to be explored in likelihood geometry, even within $\PP^3$.
  \hfill $\diamondsuit$
\end{example}

\section{Characteristic Classes}\label{Proofs}

We start by giving an alternative description of the likelihood correspondence which reveals its intimate connection with the theory of Chern classes on possibly noncompact varieties. An important role will be played by the Lie algebra and cotangent bundle of the algebraic  torus 
$(\mathbb{C}^*)^{n+1}$. This section ties our discussion to  the 
work of Aluffi  \cite{AluJSC, AluLectures, Alu} and Huh  \cite{Huh1, Huh0, Huh2}.  In particular, we introduce and explain Chern-Schwartz-MacPherson (CSM) classes. And, most importantly,
we present  proofs for Theorems \ref{thm:finite-to-one},
 \ref{thm:eulerchar},  \ref{PositiveData}, and~\ref{LinearCase}.

 Let $X \subseteq \mathbb{P}^n$ be a closed and irreducible subvariety of dimension $d$, not contained in our distinguished arrangement of $n+2$ hyperplanes,
\[
\sH=\big\{(p_0:p_1:\cdots:p_n) \in \mathbb{P}^n \mid  \, p_0 \cdot p_1 \cdots p_n \cdot p_+=0 \, \},  \qquad p_+= \sum_{i=0}^n p_i.
\]
Let $\varphi_i$ denote the restriction of the rational function $p_i/p_+$ to $X \backslash \sH$.
The  closed embedding
\[
\varphi: X \backslash \sH \longrightarrow (\mathbb{C}^*)^{n+1}, \qquad \varphi=(\varphi_0,\ldots,\varphi_n),
\]
shows that the variety $X \backslash \sH$ is {\em very affine}.
Let $x$ be a smooth point of $X \backslash \sH$. We define 
\begin{equation}
\label{eq:logarithmicGauss}
\gamma_x\,:\, T_x X \,\longrightarrow \,T_{\varphi(x)} (\mathbb{C}^*)^{n+1}
\, \longrightarrow \,\,\mathfrak{g}:=T_{1} (\mathbb{C}^*)^{n+1}
\end{equation}
 to be the derivative of $\varphi$ at $x$ followed by that of  left-translation by $\varphi(x)^{-1}$.
 Here $\mathfrak{g}$ is the Lie algebra of the algebraic torus $(\mathbb{C}^*)^{n+1}$.
In local coordinates $(x_1,\ldots,x_d)$ around the smooth point $x$, the linear map 
$\gamma_x$ is represented by the logarithmic Jacobian matrix
\[
\Bigg(\frac{\partial \log \varphi_i}{\partial x_j}\Bigg), \quad 0 \le i \le n, \quad 1 \le j \le d.
\]

The linear map $\gamma_x$ in (\ref{eq:logarithmicGauss}) is 
injective because $\varphi$ is injective. We
write $q_0,\ldots,q_n$ for the coordinate functions on
the torus $(\mathbb{C}^*)^{n+1}$. These functions define a 
$\mathbb{C}$-linear basis of the dual Lie algebra $\mathfrak{g}^\vee$ corresponding to differential forms
\[
\text{dlog} (q_0), \ldots, \text{dlog} (q_n)\,\, \in\, H^0\Big((\mathbb{C}^*)^{n+1}, \Omega^1_{(\mathbb{C}^*)^{n+1}}\Big) \,\simeq \,\mathfrak{g}^\vee \,\simeq \,\mathbb{C}^{n+1}.
\]
We fix this choice of basis of $\mathfrak{g}^\vee$, and we identify $\mathbb{P}(\mathfrak{g}^\vee)$ with the space of data vectors $\mathbb{P}^n_u$:
\[
\mathfrak{g}^\vee\, \simeq \,\Big\{  \sum_{i=0}^n u_i \cdot {\rm dlog}(q_i) \mid u =(u_0,\ldots,u_n) \in \mathbb{C}^{n+1} \Big\}.
\]
Consider the vector bundle homomorphism defined by the pullback of differential forms 
\begin{equation}
\label{eq:gamma}
\gamma^\vee:  \mathfrak{g}^\vee_{X_{\text{reg}}\backslash \sH}  \longrightarrow \Omega^1_{X_{\text{reg}}\setminus \sH}, \qquad (x,u) \longmapsto \sum_{i=0}^n u_i \cdot \text{dlog}(\varphi_i) (x).
\end{equation}
Here $\mathfrak{g}^\vee_{X_{\text{reg}}\backslash \sH} $ is the trivial vector bundle over $X_{\text{reg}}\backslash \sH$ modeled on the vector space $\mathfrak{g}^\vee$. The induced linear map $\gamma^\vee_x$ between the fibers over a smooth point $x$ is dual to the injective linear map 
$\gamma_x: T_x X \longrightarrow \mathfrak{g}$.
Therefore $\gamma^\vee$ is surjective and $\text{ker}(\gamma^\vee)$ is a vector bundle over $X_{\rm reg} \backslash \sH$. This vector bundle has positive rank $n-d+1$, and hence its projectivization is nonempty.

\begin{proof}[Proof of Theorem \ref{thm:finite-to-one}]
 Under the identification $\mathbb{P}(\mathfrak{g}^\vee) \simeq \mathbb{P}^n_u$, the projective bundle $\mathbb{P}(\ker \gamma^\vee)$ corresponds to the following constructible 
  subset of dimension $n$:
\[
\sL_X \cap\Big( (X_{\rm reg} \backslash  \sH) \times \mathbb{P}^n_u\Big) \,
\subseteq \,\mathbb{P}^n_p \times \mathbb{P}^n_u.
\]
Therefore its Zariski closure $\sL_X$ is irreducible of dimension $n$, and ${\rm pr}_1: \sL_X \to \mathbb{P}^n_p$ is  a projective bundle over $X_{\rm reg} \backslash \sH$. The 
likelihood vibration ${\rm pr}_2:\sL_X \to \mathbb{P}^n_u$ is generically finite-to-one because the domain and the range are algebraic varieties of the same dimension.
\end{proof}

Our next aim is to prove Theorem \ref{PositiveData}. For this we fix
a resolution of singularities
\[
\xymatrix{
\pi^{-1}(X_{\rm reg} \backslash \sH) \ar[r] \ar[d]& \widetilde{X} \ar[d]^\pi &\\
X_{\rm reg} \backslash \sH \ar[r]&X \ar[r]& \mathbb{P}^{n},
}
\]
where $\pi$ is an isomorphism over $X_{\rm reg} \backslash \sH$, 
the variety $\widetilde{X}$ is smooth and projective, and the complement of $\pi^{-1}(X_{\rm reg} \backslash \sH)$ is a simple normal crossing divisor in 
$\widetilde{X}$ with irreducible components
 $D_1,\ldots,D_k$.
Each $\varphi_i$ lifts to a rational function on $\widetilde{X}$ which is regular on $\pi^{-1}(X \backslash \sH)$.
If $u=(u_0,\ldots,u_n)$ is an integer vector in $\mathbb{Z}^{n+1}$, then these functions satisfy
\begin{equation}
\label{eq:orderorder}
{\rm ord}_{D_j}(\ell_u)\quad = \quad \sum_{i=0}^n u_i \cdot {\rm ord}_{D_j} (\varphi_i).
\end{equation}
If $u \in \mathbb{C}^{n+1} \backslash \mathbb{Z}^{n+1}$
 then ${\rm ord}_{D_j}(\ell_u)$ is the complex number defined by the equation (\ref{eq:orderorder}) for
 $j=1,\ldots,k$. We write $H_i:=\{p_i=0\} $ and $H_+:=\{p_+=0\}$ for the
$n+2$ hyperplanes in $\mathcal{H}$.

\begin{lemma}\label{A}
Suppose that $X \cap H_i$ is smooth along $H_+$, and let $D_j$ be a divisor in the boundary of $\widetilde{X}$ such that
$\pi(D_j) \subseteq \mathcal{H}$. 
Then the following three statements hold:
\begin{enumerate}
\item If $\,\pi(D_j) \nsubseteq H_+\,$ then
$\,{\rm ord}_{D_j} (\varphi_i)\, \ \text{is}\ \begin{cases} 
\text{positive} & \text{if $\,\pi(D_j) \subseteq H_i$,}\\ 
\text{zero}& \text{if $\pi(D_j) \nsubseteq H_i$.}
\end{cases}
$
\item If $\,\pi(D_j) \subseteq H_+\,$ then
$-{\rm ord}_{D_j} (\varphi_i)\, \ \text{is}\ \begin{cases} 
\text{positive}& \text{if $\pi(D_j) \nsubseteq H_i$,}\\
\text{nonnegative} & \text{if $\pi(D_j) \subseteq H_i$.}\\ 
\end{cases}
$ 
\item In each of the above two cases, ${\rm ord}_{D_j} (\varphi_i) $ is non-zero for at least one index $i$.
\end{enumerate}
\end{lemma}

\begin{proof}
Write $H_i'$ and $H_+'$ for the pullbacks of $H_i$ and $H_+$ to $X$ respectively. Note that $\,{\rm ord}_{D_j} (\pi^*(H_i'))$ is positive if $D_j$ is contained in $\pi^{-1}(H_i')$ and otherwise zero. Since 
\[
{\rm ord}_{D_j} (\varphi_i)\,=\,{\rm ord}_{D_j} (\pi^*(H'_i))\,-\,{\rm ord}_{D_j} (\pi^*(H'_+)),
\]
this proves the first and second assertion, except for the case when $\pi(D_j) \subseteq H_i \cap H_+$. In this case, our assumption that $H_i'$ is smooth along $H_+'$ shows that $\pi(D_j) \subseteq X_{\rm reg}$ and the order of vanishing of $H_i'$ along $\pi(D_j)$ is $1$. Therefore
\[
-{\rm ord}_{D_j} (\varphi_i)\,=\,{\rm ord}_{D_j} (\pi^*(H'_+))-1 \ge 0.
\]
The third assertion of Lemma \ref{A} is derived by the following set-theoretic reasoning:
\begin{itemize}
\item If $\pi(D_j) \nsubseteq H_+$, then $\pi(D_j) \subseteq H_i$ for some $i$ because $\pi(D_j) \subseteq \sH$ is irreducible.
\item If $\pi(D_j) \subseteq H_+$, then $\pi(D_j) \nsubseteq H_i$ for some $i$ because $\bigcap_{i=0}^n H_i = \emptyset$.  
\end{itemize}
\vskip -0.7cm
\end{proof}

From Lemma \ref{A} and equation (\ref{eq:orderorder}) we deduce the following result.
In Lemmas \ref{B} and \ref{C} we retain the hypothesis from Lemma \ref{A}
which coincides with that in Theorem \ref{PositiveData}.

\begin{lemma}\label{B}
If $\pi(D_j) \subseteq \mathcal{H}$ and $u \in \mathbb{R}^{n+1}_{> 0}$ is 
strictly positive, then ${\rm ord}_{D_j}(\ell_u)$ is nonzero.
\end{lemma}

Consider the sheaf of logarithmic differential $1$-forms
 $\,\Omega_{\widetilde{X}}^1(\log D)$, where $D$ is the sum of the irreducible components of $\pi^{-1}(\sH)$.
 If $u$ is an integer vector, then the corresponding likelihood function $\ell_u$
 on $\widetilde{X}$ defines a global section of this sheaf:
\begin{equation}
\label{eq:logdifferential}
{\rm dlog}(\ell_u)\,=\,\sum_{i=0}^n u_i \cdot {\rm dlog}(\varphi_i)
\,\, \in \,\,H^0\big(\widetilde{X}, \Omega_{\widetilde{X}}^1(\log D)\big).
\end{equation}
If $u \in \mathbb{C}^{n+1} \backslash \mathbb{Z}^{n+1}$
then we define the global section ${\rm dlog}(\ell_u)$ by the above expression (\ref{eq:logdifferential}).

\begin{lemma}\label{C}
If $u \in \mathbb{R}^{n+1}_{>0}$ is strictly positive, then ${\rm dlog}(\ell_u)$ does not vanish on $\pi^{-1}(\mathcal{H})$.
\end{lemma}

\begin{proof}
Let  $x \in \pi^{-1}(\sH)$ and $D_1,\ldots,D_l$ the irreducible components of $D$ containing $x$,
with local equations $g_1,\ldots,g_l$ on a small neighborhood $G$ of $x$. 
Clearly, $l \geq 1$.
By passing to a smaller neighborhood if necessary, we may assume that $\Omega^1_{\widetilde{X}}(\log D)$ trivializes over $G$, and 
\[
{\rm dlog}(\ell_u) \quad = \quad \sum_{j=1}^l {\rm ord}_{D_j}(\ell_u) \cdot {\rm dlog}(g_j) \,\,+\,\,\psi,
\]
where $\psi$ is a regular $1$-form. Since the ${\rm dlog}(g_j)$ form part of a free basis of a trivialization of $\Omega_{\widetilde{X}}^1(\log D)$ over $G$,  Lemma \ref{B}
implies that ${\rm dlog}(\ell_u)$ is nonzero on $\pi^{-1}(\mathcal{H})$ if $ u \in \mathbb{R}^{n+1}_{>0}$.
\end{proof}

\begin{proof}[Proof of Theorem \ref{thm:eulerchar}]
In the notation above, the logarithmic Poincar\'e-Hopf theorem states 
\[
\int_{\widetilde{X}}c_d\big(\Omega^1_{\widetilde{X}}(\log D)\big)
\,\,= \,\, (-1)^d \cdot \chi\big(\widetilde{X} \backslash \pi^{-1}(\sH)\big).
\]
See \cite[Section 3.4]{AluLectures} for example.
If $X \backslash \sH$ is smooth, then Lemma \ref{C} shows that, for generic $u$, the zero-scheme of the section (\ref{eq:logdifferential}) is equal to the likelihood locus
\[
\big\{x \in X \backslash \sH \mid {\rm dlog}(\ell_u)(x)=0\big\}.
\]
Since the likelihood locus is a zero-dimensional scheme of length equal to the ML degree of $X$, the logarithmic Poincar\'e-Hopf theorem implies Theorem \ref{thm:eulerchar}.
\end{proof}

\begin{proof}[Proof of Theorem \ref{PositiveData}]
Suppose that the likelihood locus
$\,
\{x \in X_{\rm reg} \backslash \sH \mid {\rm dlog}(\ell_u)(x)=0\}
\,$
contains a curve. Let $C$ and $\widetilde{C}$ denote the closures of that curve in $X$ and $\widetilde{X}$ respectively.
Let $\pi^*(\mathcal{H})$ be the pullback of the divisor $\mathcal{H} \cap X$ of $X$. 
If $u \in \mathbb{R}^{n+1}_{>0}$ then Lemma \ref{C} 
implies  that $\, \pi^*(\mathcal{H}) \cdot \widetilde{C}$ is rationally equivalent to zero in $\widetilde{X}$.
It then follows from the Projection Formula that
$\, \mathcal{H} \cdot C$  is also rationally equivalent to zero in $\mathbb{P}^n$.
But this is impossible. Therefore the likelihood locus does not contain a curve. This proves the first part of Theorem \ref{PositiveData}.

For the second part, we first show that ${\rm pr}_2^{-1}(u)$ is contained in $X \backslash \sH$ for a strictly positive vector $u$. This means  there is no 
pair $(x,u) \in \mathcal{L}_X$ with $x \in \sH$ which is a limit of the form
\[
(x,u)=\lim_{t \to 0} \, (x_t,u_t), \qquad x_t \in X_{\rm reg} \backslash \sH, \qquad 
{\rm dlog}(\ell_{u_t})(x_t)=0.
\]
If there is such a sequence $(x_t,u_t)$, then we can take its limit over $\widetilde{X}$ to find a point $\widetilde{x} \in \widetilde{X}$ such that ${\rm dlog}(\ell_{u})(\widetilde{x})=0$, 
but this would contradict Lemma \ref{C}. 

Now suppose that  the fiber ${\rm pr}_2^{-1}(u)$ is contained in $X_{\rm reg}$, and hence in $X_{\rm reg} \backslash \sH$. 
By Theorem \ref{thm:finite-to-one}, this fiber ${\rm pr}^{-1}_2(u)$ is contained the smooth variety $(\sL_X)_{\rm reg}$. Furthermore, by the first part of Theorem \ref{PositiveData}, ${\rm pr}_2^{-1}(u)$ is 
a zero-dimensional subscheme of $(\sL_X)_{\rm reg}$.
 The assertion on the length of the fiber now follows from a standard result on intersection theory on Cohen-Macaulay varieties. More precisely, we have
\begin{eqnarray*}
{\rm MLdegree}(X)\,=\,(U_1 \cdot \ldots \cdot U_n)_{\sL_X}
\,=\,(U_1 \cdot \ldots \cdot U_n)_{(\sL_X)_{\rm reg}}
\,=\,\deg({\rm pr}_2^{-1}(u)),
\end{eqnarray*}
where the $U_i$ are pullbacks of sufficiently general hyperplanes in $\mathbb{P}^n_u$ containing $u$, and the two terms in the middle are the intersection numbers defined in \cite[Definition 2.4.2]{Fulton}. The fact that $(\sL_X)_{\rm reg}$ is Cohen-Macaulay is used in the last equality \cite[Example 2.4.8]{Fulton}.
\end{proof}

\begin{remark} \rm
If $X$ is a curve, then the zero-scheme of the section (\ref{eq:logdifferential}) is zero-dimensional for generic $u$, even if $X \backslash \sH$ is singular. Furthermore,  the length of this zero-scheme is at least as large as ML degree of $X$. Therefore
\[
-\chi(X \backslash \sH) \,\ge\,  -\chi\big(\widetilde{X} \backslash \pi^{-1}(\sH)\big)
\, \ge \,{\rm MLdegree}\,(X).
\]
This proves that Conjecture \ref{conj:eulerpositive} holds for $d=1$.
\end{remark}

Next we give a brief description of  the Chern-Schwartz-MacPherson (CSM) 
class. For a gentle introduction we refer to \cite{AluLectures}.
The group  $C(X)$  of {\em constructible functions} on a complex algebraic variety $X$ is 
a subgroup of the group of integer valued functions on $X$. It is generated by the
characteristic functions $\mathbf{1}_Z$ of all closed subvarieties
$Z$ of $X$.  If $f: X \to Y$ is a morphism between complex algebraic varieties, then the pushforward of constructible functions is the homomorphism
\[
f_* : C(X) \longrightarrow C(Y), \qquad \mathbf{1}_Z \longmapsto \Big( y \longmapsto \chi\big(f^{-1}(y) \cap Z \big), \quad y \in Y \Big).
\]
If $X$ is a compact complex manifold, then the characteristic class of $X$ is the Chern class of the tangent bundle $c(TX) \cap [X] \in H_*(X;\mathbb{Z})$. A generalization to possibly singular or noncompact varieties is provided by the Chern-Schwartz-MacPherson class, whose existence was once a conjecture of Deligne and Grothendieck. 

In the next definition, we write  $C$ for the functor of constructible functions from the category of complete complex algebraic varieties to the category of abelian groups.

\begin{definition}\label{def:CSM} \rm
The \emph{CSM class} is the unique natural transformation
\[
c_{SM} : C \longrightarrow H_*
\]
such that 
$
c_{SM}(\mathbf{1}_X) = c(TX) \cap [X] \in H_*(X;\mathbb{Z})
$ 
when $X$ is smooth and complete.
\end{definition}

The uniqueness follows from the naturality, the resolution of singularities
over $\mathbb{C}$, and the requirement for smooth and complete varieties. 
We highlight two properties of the CSM class  which follow directly from Definition \ref{def:CSM}:
\begin{enumerate}
\item The CSM class satisfies the inclusion-exclusion relation
\begin{equation}\label{eq:inclusionexclusion}
c_{SM}(\mathbf{1}_{U \cup U'})=c_{SM}(\mathbf{1}_{U})+c_{SM}(\mathbf{1}_{U'})-c_{SM}(\mathbf{1}_{U \cap U'}) \in H_*(X;\mathbb{Z}).
\end{equation}
\item The CSM class captures the topological Euler characteristic as its degree:
\begin{equation}\label{eq:chi}
\chi(U) = \int_X c_{SM}(\mathbf{1}_U) \in \mathbb{Z}.
\end{equation}
\end{enumerate}
Here $U$ and $U'$ are arbitrary constructible subsets of a complete variety $X$.

What kind of information on a constructible subset is encoded in its CSM class? 
In likelihood geometry,  $U$ is a constructible subset
in  the complex projective space $\PP^n$, and we identify
$c_{SM}(\mathbf{1}_U) $ with its image in
$\,H_*(\PP^n,\mathbb{Z}) = \mathbb{Z}[p]/\langle p^{n+1} \rangle$.
Thus $c_{SM}(\mathbf{1}_U)$ is a polynomial
of degree $\leq n$ is one variable $p$. To be consistent
with the earlier sections, we introduce a homogenizing variable $u$,
and we write $c_{SM}(\mathbf{1}_U)$ as a binary form 
of degree $n$ in $(p,u)$.

The CSM class of $U$ carries the same information as
the {\em sectional Euler characteristic}
$$ \chi_{\rm sec}(\mathbf{1}_U)
\quad = \quad \sum_{i=0}^n \chi( U \cap L_{n-i}) \cdot p^{n-i} u^i .
$$
Here $L_{n-i}$ is a generic linear subspace
of codimension $i$ in $\PP^n$.
Indeed, it was proved by Aluffi in \cite[Theorem 1.1]{Alu}
that  $c_{SM}(\mathbf{1}_U) $  is the transform of
$ \chi_{\rm sec}(\mathbf{1}_U)$ under a linear involution
on binary forms of degree $n$ in $(p,u)$.
In fact, our involution in Conjecture \ref{conj:aluffi}
is nothing but the signed version of the Aluffi's involution.
This is explained by the following result.

\begin{theorem}\label{proCSM2}
Let $X \subset \PP^n$ be closed subvariety of dimension $d$
that is not contained in $\sH$.
If the very affine variety $X \backslash \sH$ is sch\"on then,
up to signs,
the ML bidegree  equals the CSM class 
and the sectional ML degree equals the sectional Euler characteristic.
In symbols,
$$  c_{SM}(\mathbf{1}_{X\backslash \sH}) 
\,\,=\,\, (-1)^{n-d} \cdot B_X(-p,u) 
\quad \hbox{and} \quad\,
\chi_{\rm sec}(\mathbf{1}_{X\backslash \sH}) 
\,\,=\,\, (-1)^{n-d} \cdot S_X(-p,u) .
$$
\end{theorem}

\begin{proof}
The first identity is a special case of
\cite[Theorem 2]{Huh1}, 
here adapted to $\PP^n$ minus $n+2$ hyperplanes, and the second identity follows from
the first by way of \cite[Theorem 1.1]{Alu}.
\end{proof}

To make sense of the statement in
Theorem \ref{proCSM2}, we need to recall the definition of {\em sch\"on}.
This term was coined by Tevelev in his study of
tropical compactifications \cite{Tevelev}. Let $U$ be an arbitrary closed 
subvariety of the algebraic torus  $(\mathbb{C}^*)^{n+1}$. In our application,
$U = X \backslash \mathcal{H}$.
We consider the closures $\overline{U}$ of $U$ in various (not necessarily complete) normal toric varieties $Y$ with dense torus
$(\mathbb{C}^*)^{n+1}$.
 The closure $\overline{U}$ is complete if and only if the support of the fan of $Y$ contains the tropicalization of $U$ \cite[Proposition 2.3]{Tevelev}. We say that $\overline{U}$ is a \emph{tropical compactification} of $U$ if it is complete and the multiplication map
\[
m\,:\,(\mathbb{C}^*)^{n+1}
 \times \overline{U} \longrightarrow Y, \quad (t,x) \longmapsto t \cdot x
\]
is flat and surjective. Tropical compactifications exist, and they are obtained from toric varieties $Y$ defined by sufficiently fine fan structures on the tropicalization of $U$ \cite[\S 2]{Tevelev}. 
The very affine variety $U$ is called
 \emph{sch\"on} if the multiplication is smooth for some tropical compactification of $U$.
Equivalently, $U$ is sch\"on if the multiplication is smooth for every tropical compactification of $U$,
by \cite[Theorem 1.4]{Tevelev}.

Two classes of sch\"on very affine varieties  are of particular interest. The first is the class of complements of essential hyperplane arrangements.
The second is the class of nondegenerate hypersurfaces.
What we need from the sch\"on hypothesis is the existence of a simple normal crossings compactification which admits sufficiently many differential one-forms which have logarithmic singularities along the boundary. For complements of hyperplane arrangements, such a compactification is provided by the wonderful compactification of De Concini and Procesi \cite{DP}. For nondegenerate hypersurfaces, and more generally for nondegenerate complete intersections, the needed compactification has been constructed by Khovanskii \cite{Hovanskii}.

We illustrate this in the setting of likelihood geometry by a $d$-dimensional linear subspace of $X \subset \mathbb{P}^n$.
The intersection of $X$ with distinguished hyperplanes $\sH$ of $\mathbb{P}^n$ is an arrangement of $n+2$ hyperplanes in $X \simeq \mathbb{P}^d$, defining a matroid $M$ of rank $d+1$ on $n+2$ elements.

\begin{proposition}\label{characteristic}
If $X$ is a linear space of dimension $d$ then
the CSM class of $X \backslash \sH$ in $\mathbb{P}^n$~is 
\[
c_{SM}(\mathbf{1}_{X \backslash \sH})\,\,=\,\,\sum_{i=0}^{d} (-1)^i h_i  u^{d-i} p^{n-d+i}.
\]
where the $h_i$ are the signed coefficients of the shifted characteristic polynomial in
(\ref{eq:chih}).
\end{proposition}

\begin{proof} 
This holds because the recursive formula for a triple of arrangement complements
\[
c_{SM}(\mathbf{1}_{U_1})\,=\,
c_{SM}(\mathbf{1}_{U}-\mathbf{1}_{U_0})\,=\,c_{SM}(\mathbf{1}_{U})-c_{SM}(\mathbf{1}_{U_0}),
\]
agrees with the usual deletion-restriction formula \cite[Theorem 2.56]{OTBook}:
\[
\chi_{M_1}(q+1)\,\,= \,\,\chi_M(q+1)-\chi_{M_0}(q+1).
\]
Here our notation is as in \cite[\S 3]{Huh1}.
We now use induction on the number of hyperplanes.
\end{proof}

\begin{proof}[Proof of Theorem \ref{LinearCase}]
The very affine variety  $X \backslash \sH$ is sch\"on when $X$ is linear.
Hence the asserted formula 
for the ML bidegree of $X$ follows from
Theorem \ref{proCSM2} and Proposition \ref{characteristic}.
\end{proof}

Rank constraints on matrices
are important both in statistics and
in algebraic geometry, and they provide a rich source
of test cases for the theory developed here.
We close our discussion with the enumerative invariants
of three hypersurfaces defined by $3 \times 3$-determinants.
It would be very interesting to 
compute these formulas for  larger
determinantal varieties.

\begin{example} \label{ex:manyclasses} \rm
We record the {\em ML bidegree}, the
{\em CSM class},
the {\em sectional ML degree}, and the {\em sectional Euler characteristic}
for three singular hypersurfaces seen earlier in this paper.
These examples were studied already in \cite{HKS}. The classes we present
are elements of $H^*(\mathbb{P}^n_p \times \mathbb{P}^n_u)$ and of $H^*(\mathbb{P}^n_p;\mathbb{Z})$ respectively, and they are written as binary forms in $(p,u)$ as before.
\begin{itemize}
\item The $3 \times 3$ determinantal hypersurface in $\mathbb{P}^8$ (Example \ref{ex:33determinant}) has
\begin{eqnarray*}
B_X(p,u)&=& \phantom{-}10p^8+24p^7u+33p^6u^2+38p^5u^3+39p^4u^4+33p^3u^5+12p^2u^6+3pu^7,\\
c_{SM}(\mathbf{1}_{X \backslash \sH})&=&-11p^8+26p^7u-37p^6u^2+44p^5u^3-45p^4u^4+33p^3u^5-12p^2u^6+3pu^7, \\
S_X(p,u)&=&\!\! \phantom{-}11p^8+ 182 p^7u+ 436 p^6u^2{+} 518 p^5u^3{+} 351 p^4u^4{+} 138 p^3u^5
{+} 30 p^2u^6{+}3pu^7,\\
\chi_{\rm sec}(\mathbf{1}_{X \backslash\sH})&=& \!\!
-11p^8+ 200 p^7u- 470 p^6u^2{+} 542 p^5u^3{-} 357 p^4u^4{+} 138 p^3u^5{-} 30 p^2u^6{+}3pu^7.
\end{eqnarray*}
\item The $3 \times 3$ symmetric determinantal hypersurface in $\mathbb{P}^5$ (Example \ref{ex:galois33}) has
\begin{eqnarray*}
B_X(p,u) &=&   6 p^5+12 p^4 u+15 p^3 u^2+12 p^2 u^3+3 p u^4,\\
c_{SM}(\mathbf{1}_{X \backslash \sH}) &=& 7 p^5-14 p^4 u+19 p^3 u^2-12 p^2 u^3+3 p u^4, \\
S_X(p,u) &=&   6 p^5+ 42 p^4 u+48 p^3 u^2+ 21 p^2 u^3+3 p u^4,\\
\chi_{\rm sec}(\mathbf{1}_{X \backslash \sH}) &=& 
7 p^5- 48 p^4 u+52  p^3 u^2-21 p^2 u^3+3 p u^4.
\end{eqnarray*}
\item The secant variety of the rational normal curve in $\mathbb{P}^4$ (Example \ref{ex:secant}) has
\begin{eqnarray*}
B_X(p,u) &=& 
\phantom{-}12 p^4 + 15 p^3 u + 12 p^2 u^2 + 3 p u^3,\\
c_{SM}(\mathbf{1}_{X \backslash \sH}) & = &
-13 p^4 + 19 p^3 u - 12 p^2 u^2 + 3 p u^3, \\
S_X(p,u) &=& \phantom{-} 12 p^4 + 30 p^3 u + 18 p^2 u^2 + 3 p u^3,\\
\chi_{\rm sec}(\mathbf{1}_{X \backslash \sH}) & = &
- 13 p^4 + 34 p^3 u - 18 p^2 u^2 + 3 p u^3 .
\end{eqnarray*}
\end{itemize}
In all known examples, the coefficients of $B_X(p,u)$ are less than or equal to  the absolute value of the corresponding coefficients of $c_{SM}(\mathbf{1}_{X \backslash \sH})$,
and similary for  $S_X(p,u)$ and $\chi_{\rm sec}(\mathbf{1}_{X \backslash \sH})$.
 That this inequality holds for the first coefficient is Conjecture \ref{conj:eulerpositive} 
 which relates the  ML degree of a singular $X$ to the
  signed Euler characteristic of the very affine variety $X \backslash \mathcal{H}$.
 \hfill $\diamondsuit$
\end{example}

\medskip

\noindent {\bf Acknowledgments}:
We thank Paolo Aluffi and Sam Payne for helpful communications,
and the Mathematics Department at KAIST, Daejeon,
for hosting both authors in May 2013.
Bernd Sturmfels was supported by 
NSF (DMS-0968882) and DARPA (HR0011-12-1-0011).

\bigskip
\bigskip
\bigskip

\noindent
Authors' addresses:

\medskip

\noindent
June Huh,
Department of Mathematics, University of Michigan, Ann Arbor,
MI 48109, USA, {\tt {junehuh@umich.edu}}

\medskip

\noindent
Bernd Sturmfels,
Department of Mathematics, University of California, Berkeley,  CA 94720, USA,
 {\tt {bernd@berkeley.edu}}

\end{document}